\documentclass[article]{siamart190516}
\usepackage{url}
\usepackage{amssymb}
\usepackage{amsmath}
\usepackage{stmaryrd}
\usepackage{siunitx}
\usepackage{enumitem}
\usepackage{commath}
\usepackage{graphicx,subfig}
\usepackage[numbers,square]{natbib}
\makeatletter
\newcommand{\tnorm}{\@ifstar\@tnorms\@tnorm}
\newcommand{\@tnorms}[1]{%
  \left|\mkern-1.5mu\left|\mkern-1.5mu\left|
   #1
  \right|\mkern-1.5mu\right|\mkern-1.5mu\right|
}
\newcommand{\@tnorm}[2][]{%
  \mathopen{#1|\mkern-1.5mu#1|\mkern-1.5mu#1|}
  #2
  \mathclose{#1|\mkern-1.5mu#1|\mkern-1.5mu#1|}
}
\makeatother

\newcommand{\jump}[1]{\llbracket #1 \rrbracket}

\newsiamremark{remark}{Remark}

\DeclareMathOperator{\Ker}{Ker}
\DeclareMathOperator{\Image}{Im}

\title{Preconditioning for a pressure-robust HDG discretization of the Stokes equations
  \thanks{\funding{SR gratefully acknowledges support from the Natural
      Sciences and Engineering Research Council of Canada through the
      Discovery Grant program (RGPIN-05606-2015).}}}
\author{Sander Rhebergen\thanks{Department of Applied Mathematics,
    University of Waterloo, Canada (\url{srheberg@uwaterloo.ca}),
    \url{http://orcid.org/0000-0001-6036-0356}} \and
  Garth~N.~Wells\thanks{Department of Engineering, University of
    Cambridge, United Kingdom (\url{gnw20@cam.ac.uk}),
    \url{https://orcid.org/0000-0001-5291-7951}}}
\headers{Preconditioning HDG for the Stokes equations}{Sander
  Rhebergen and Garth N. Wells}
\begin{document}
\maketitle
\begin{abstract}
  We introduce a new preconditioner for a recently developed
  pressure-robust hybridized discontinuous Galerkin (HDG) finite
  element discretization of the Stokes equations. A feature of HDG
  methods is the straightforward elimination of degrees-of-freedom
  defined on the interior of an element. In our previous work
  (J. Sci. Comput., 77(3):1936--1952, 2018) we introduced a
  preconditioner for the case in which only the degrees-of-freedom
  associated with the element velocity were eliminated via static
  condensation. In this work we introduce a preconditioner for the
  statically condensed system in which the element pressure
  degrees-of-freedom are also eliminated. In doing so the number of
  globally coupled degrees-of-freedom are reduced, but at the expense
  of a more difficult problem to analyse. We will show, however, that
  the Schur complement of the statically condensed system is
  spectrally equivalent to a simple trace pressure mass matrix. This
  result is used to formulate a new, provably optimal
  preconditioner. Through numerical examples in two- and
  three-dimensions we show that the new preconditioned iterative
  method converges in fewer iterations, has superior conservation
  properties for inexact solves, and is faster in CPU time when
  compared to our previous preconditioner.
\end{abstract}
\begin{keywords}
  Stokes equations, preconditioning, hybridized, discontinuous
  Galerkin, finite element methods.
\end{keywords}
\begin{AMS}
  65F08, 
  65N30, 
  76D07. 
\end{AMS}
\section{Introduction}
\label{sec:introduction}

Hybridizable discontinuous Galerkin (HDG) methods were introduced for
elliptic problems to reduce cost of discontinuous Galerkin methods while
retaining favourable properties~\cite{Cockburn:2009a}. Static
condensation can be applied to eliminate degrees-of-freedom on cells,
leading to global degrees-of-freedom that are associated with functions
that are defined only on the facets of the mesh. The resulting methods
have fewer global degrees of freedom than a discontinuous Galerkin
method on the same mesh, especially in three dimensions.

For large systems, preconditioned iterative methods are preferred over
direct solvers. In the case of HDG methods, new preconditioners need to
be designed that are effective for the reduced linear systems following
static condensation. This has been a topic of recent research with the
design of scalable multigrid and domain decomposition methods for HDG
discretizations of elliptic PDEs, see for example \cite{cockburn:2014,
Gander:2015, Kronbichler:2018, Muralikrishnan:2020}. See also
\cite{Muralikrishnan:2017, Muralikrishnan:2018} for HDG preconditioners
for hyperbolic problems and \cite{Dobrev:2019} and \cite{Wildey:2019}
for preconditioning of statically condensed and hybridized finite
element discretizations by, respectively, algebraic multigrid and
geometric multigrid.

In this paper we develop a new preconditioner for the pressure-robust
HDG discretization of the Stokes equations introduced in
\cite{Rhebergen:2018a, Rhebergen:2017}. Pressure-robust discretizations
have the advantage over other finite element discretizations that
the~\emph{a priori} error estimates for the velocity do not depend on
the pressure. We refer to, respectively, \cite{John:2017} and
\cite{Lederer:2017, Lederer:2018, Lederer:2019, Lehrenfeld:2016} for an
overview of pressure-robust discretizations and alternative
pressure-robust HDG discretizations. In~\cite{Rhebergen:2018b} we
developed a preconditioner for a pressure-robust HDG discretization of
the Stokes problem, and showed that the pressure Schur complement of the
linear system obtained after eliminating only the degrees-of-freedom
associated to the element velocity is identical to the pressure Schur
complement of the original linear system. Together with a proof showing
spectral equivalence of the element/trace pressure Schur complement with
an element/trace pressure mass matrix, we were able to formulate a very
simple scalable preconditioner when the velocity only was statically
condensed. However, a disadvantage of eliminating the cell-wise velocity
degrees-of-freedom and not the cell-wise pressure degrees-of-freedom is
that the pointwise divergence-free (within cells) property of the method
is guaranteed only in the limit of the iterative solver convergence. We
address this issue in this work by formulating a preconditioner for a
statically condensed system after elimination of \emph{both} the
degrees-of-freedom associated to the element velocity \emph{and} the
degrees-of-freedom associated to the element pressure. The complication
that arises in eliminating the degrees-of-freedom associated to the
element pressure is that the lifting of the trace velocity unknowns to
the element is divergence-free. This requires a new proof to show
spectral equivalence of the trace pressure Schur complement with the
trace pressure mass matrix. Given this result we can then apply the
general theory of Pestana and Wathen \cite{Pestana:2015} for
preconditioners of saddle point problems to formulate a new scalable
preconditioner for the statically condensed problem. Compared to our
previous work \cite{Rhebergen:2018b}, there are less globally coupled
degrees-of-freedom and our new preconditioner results in a more
efficient solver.

The outline of this paper is as follows. In \cref{sec:hdg_stokes} we
present the HDG method for the Stokes equations and discuss some useful
properties of the discretization. A preconditioner for the statically
condensed form of the HDG method is formulated in
\cref{sec:preconditioning}. By two- and three-dimensional numerical
examples in \cref{sec:num_examples} we verify our analysis. We draw
conclusions in \cref{sec:conclusions}.

\section{HDG for the Stokes problem}
\label{sec:hdg_stokes}

Let $\Omega\subset\mathbb{R}^d$ be a polygonal ($d = 2$) or polyhedral
($d = 3$) domain with boundary $\partial\Omega$. We consider the Stokes
problem: given a body force $f : \Omega \to \mathbb{R}^d$, find the
velocity $u : \Omega \to \mathbb{R}^d$ and (kinematic) pressure $p :
\Omega \to \mathbb{R}$ such that
\begin{subequations}
  \label{eq:stokes}
  \begin{align}
    -\nabla^2u + \nabla p &= f & & \text{in } \Omega,
    \\
    \nabla\cdot u &= 0 & & \text{in } \Omega,
    \\
    u &= 0 & & \text{on } \partial\Omega.
  \end{align}
\end{subequations}
To obtain a unique solution to the pressure, we additionally impose that
the mean of the pressure over the domain $\Omega$ is zero.

\subsection{Preliminaries}

Consider a triangulation $\mathcal{T} := \{K\}$ of the domain $\Omega$
into non-overlapping elements $K$. We introduce the following function
spaces on $\Omega$:
\begin{equation}
  \begin{split}
    V_h  &:= \cbr[1]{v_h\in \sbr[0]{L^2(\Omega)}^d
        : \ v_h \in \sbr[0]{P_k(K)}^d, \ \forall\ K\in\mathcal{T}},
    \\
    Q_h &:= \cbr[1]{q_h\in L^2(\Omega) : \ q_h \in P_{k-1}(K) ,\
    \forall \ K \in \mathcal{T}},
  \end{split}
  \label{eqn:spaces_cell}
\end{equation}
where $P_k(D)$ denotes the set of polynomials of degree at most~$k$ on a
domain~$D$. Two adjacent elements $K^+$ and $K^-$ share an interior
facet $F$. A boundary facet is defined as a facet of the boundary of an
element, $\partial K$, that lies on $\partial\Omega$. Denoting the set
of all facets by $\mathcal{F} = \{F\}$, and the union of all facets by
$\Gamma^0$, we introduce the following function spaces on $\Gamma^0$:
\begin{equation}
  \begin{split}
    \bar{V}_h &:= \cbr[1]{\bar{v}_h \in \sbr[0]{L^2(\Gamma^0)}^d:\ \bar{v}_h \in
      \sbr[0]{P_{k}(F)}^d\ \forall\ F \in \mathcal{F},\ \bar{v}_h
      = 0 \ \mbox{on}\ \partial\Omega},
    \\
    \bar{Q}_h &:= \cbr[1]{\bar{q}_h \in L^2(\Gamma^0) : \ \bar{q}_h \in
      P_{k}(F) \ \forall\ F \in \mathcal{F}}.
  \end{split}
  \label{eqn:spaces_facet}
\end{equation}
For convenience, we introduce the spaces
$\boldsymbol{V}_h := V_h \times \bar{V}_h$,
$\boldsymbol{Q}_h := Q_h \times \bar{Q}_h$, and
$\boldsymbol{X}_h := \boldsymbol{V}_h \times
\boldsymbol{Q}_h$. Function pairs in $\boldsymbol{V}_h$ and
$\boldsymbol{Q}_h$ will be denoted by boldface, e.g.,
$\boldsymbol{v}_h := (v_h, \bar{v}_h) \in \boldsymbol{V}_h$ and
$\boldsymbol{q}_h := (q_h, \bar{q}_h) \in \boldsymbol{Q}_h$.

For scalar-valued functions $p$ and $q$, we write
\begin{equation}
  (p, q)_{\mathcal{T}} := \sum_{K \in \mathcal{T}}(p, q)_K, \qquad
  \langle p, q \rangle_{\partial\mathcal{T}} := \sum_K \langle p, q \rangle_{\partial K},
\end{equation}
where $(p, q)_K := \int_{K} p q \dif x$ and
$\langle p, q \rangle_{\partial K} := \int_{\partial K} p q \dif
s$. Similar inner-products hold for vector-valued functions.

We also introduce the following mesh-dependent norms:
\begin{subequations}
  \begin{align}
    \label{eq:def_norm_facet_L2}
    \norm{\bar{v}_h}_{h,0}^2 &:= \sum_{K\in\mathcal{T}_h}h_{K}
                               \norm{\bar{v}_h}_{\partial K}^2
    && \forall \bar{v}_h \in \bar{V}_h,
    \\
    \label{eq:def_norm_hK_ap}
    \tnorm{\bar{v}_h}_{h}^2 &:= \sum_{K\in\mathcal{T}_h}h_{K}^{-1}
                              \norm{\bar{v}_h - m_K(\bar{v}_h)}_{\partial K}^2
    && \forall \bar{v}_h \in \bar{V}_h,
    \\
    \label{eq:stability_norm}
    \tnorm{ \boldsymbol{v}_h }_v^2 &:= \sum_{K\in\mathcal{T}}\norm{\nabla v_h }^2_{K}
                              + \sum_{K\in\mathcal{T}} \alpha  h_K^{-1} \norm{\bar{v}_h - v_h}^2_{\partial K}
    && \forall \boldsymbol{v}_h \in \boldsymbol{V}_h,
    \\
    \label{eq:def_normbarq}
    \norm{\bar{q}_h}_{p}^2 &:= \sum_{K\in\mathcal{T}}h_K\norm{\bar{q}_h}^2_{\partial K}
    && \forall \bar{q}_h \in \bar{Q}_h,
    \\
    \label{eq:def_normq}
    \tnorm{\boldsymbol{q}_h}^2_p &:= \norm{q_h}^2_{\Omega} + \norm{\bar{q}_h}_{p}^2
    && \forall \boldsymbol{q}_h \in \boldsymbol{Q}_h,
  \end{align}
\end{subequations}
where $h_K$ is the length measure of an element $K$, $\alpha > 0$ is a
constant, and
$m_K(\bar{v}_h) := \frac{1}{|\partial K|}\int_{\partial K} \bar{v}_h
\dif s$. We furthermore have the following Poincar\'e-type inequality
(see \cite[Lemma 3.7]{cockburn:2014} and \cite[Proof of
Theorem~2.3]{Gopalakrishnan:2003}) for the norms on $\bar{V}_h$:
\begin{equation}
  \label{eq:Poincare_facet}
  \norm{\bar{v}_h}_{h,0} \le c \tnorm{\bar{v}_h}_h \quad \forall \bar{v}_h \in \bar{V}_h.
\end{equation}

We will use the following reduced version of
\cite[Theorem~3.1]{Howell:2011}.
\begin{theorem}
  \label{thm:howell}
  Let $U$, $P_1$, and $P_2$ be reflexive Banach spaces, and let
  $b_1:P_1\times U \to \mathbb{R}$ and $b_2:P_2\times U \to \mathbb{R}$
  be bilinear and continuous. Let
  \begin{equation}
    Z_{b_i} = \cbr{v \in U\ :\ b_i(p_i, v) = 0\ \forall p_i \in P_i} \subset U, \quad i=1,2,
  \end{equation}
  then the following are equivalent:
  \begin{enumerate}[label=\roman*.]
  \item There exists $c > 0$ such that
    \begin{equation*}
      \sup_{v \in U} \frac{b_1(p_1, v) + b_2(p_2, v)}{\norm{v}_U}
      \ge c\del[1]{\norm{p_1}_{P_1} + \norm{p_2}_{P_2}}
      \quad (p_1, p_2)\in P_1\times P_2.
    \end{equation*}
  \item There exists $c > 0$ such that
    \begin{equation*}
      \sup_{v \in U} \frac{b_1(p_1, v)}{\norm{v}_U}
      \ge c\norm{p_1}_{P_1},\ p_1\in P_1
      \ \text{and}\
      \sup_{v \in Z_{b_1}} \frac{b_2(p_2, v)}{\norm{v}_U}
      \ge c\norm{p_2}_{P_2},\ p_2\in P_2.
    \end{equation*}
  \end{enumerate}
\end{theorem}

\subsection{The HDG formulation of the Stokes problem}

We consider the HDG method of~\cite{Labeur:2012, Rhebergen:2018a,
Rhebergen:2017} for the Stokes problem \cref{eq:stokes}, which reads:
find $(\boldsymbol{u}_h, \boldsymbol{p}_h) \in \boldsymbol{X}_h$ such
that
\begin{equation}
  \label{eq:discrete_problem}
  a_h(\boldsymbol{u}_h, \boldsymbol{v}_h)
  + b_h(\boldsymbol{p}_h, v_h)
  + b_h(\boldsymbol{q}_h, u_h)
  = \del{v_h, f}_{\mathcal{T}}
  \quad \forall (\boldsymbol{v}_h, \boldsymbol{q}_h) \in \boldsymbol{X}_h,
\end{equation}
where
\begin{subequations}
  \begin{align}
    \label{eq:formA}
    a_h(\boldsymbol{w}_h, \boldsymbol{v}_h)
    :=&
    \del{\nabla w_h, \nabla v_h}_{\mathcal{T}}
    + \langle \alpha h^{-1}(w_h - \bar{w}_h), v_h - \bar{v}_h \rangle_{\partial\mathcal{T}}
    \\
    \nonumber
    & - \left\langle w_h - \bar{w}_h, \partial_n v_h \right\rangle_{\partial\mathcal{T}}
      - \left\langle \partial_n w_h, v_h - \bar{v}_h \right\rangle_{\partial\mathcal{T}},
    \\
    \label{eq:formB}
    b_h(\boldsymbol{q}_h, v_h)
    :=&
    - \del{q_h, \nabla \cdot v_h}_{\mathcal{T}}
    + \left\langle v_h \cdot n, \bar{q}_h \right\rangle_{\partial \mathcal{T}},
  \end{align}
  \label{eq:bilin_forms}
\end{subequations}
and where $n$ is the outward unit normal vector on $\partial K$.

The following properties of the bilinear forms will be useful when
constructing a preconditioner for the statically condensed version
of~\cref{eq:discrete_problem}. For sufficiently large $\alpha$,
$a_h(\cdot, \cdot)$ is coercive on $\boldsymbol{V}_h$ and bounded, i.e.,
there exist constants $c_a^s > 0$ and $c_a^b > 0$, independent of $h$,
such that for all $\boldsymbol{u}_h, \boldsymbol{v}_h \in
\boldsymbol{V}_h$,
\begin{equation}
  \label{eq:stab_bound_ah}
  a_h(\boldsymbol{v}_h, \boldsymbol{v}_h) \ge c_a^s\tnorm{\boldsymbol{v}_h}_v^2 \qquad\mbox{and}\qquad
  \envert{a_h(\boldsymbol{u}_h, \boldsymbol{v}_h)} \le c_a^b\tnorm{\boldsymbol{u}_h}_v\tnorm{\boldsymbol{v}_h}_v,
\end{equation}
see \cite[Lemmas~4.2 and~4.3]{Rhebergen:2017}. Furthermore, there
exist constants $c_b^b > 0$ and $\beta_p > 0$, independent of $h$,
such that for all $\boldsymbol{v}_h \in \boldsymbol{V}_h$ and for all
$\boldsymbol{q}_h \in \boldsymbol{Q}_h$,
\begin{equation}
  \label{eq:bound_bh}
  \envert{b_h(\boldsymbol{q}_h, v_h)} \le c_b^b \tnorm{\boldsymbol{v}_h}_v\tnorm{\boldsymbol{q}_h}_p
  \qquad\mbox{and}\qquad
    \beta_p \tnorm{\boldsymbol{q}_h}_{p} \le \sup_{\boldsymbol{v}_h\in \boldsymbol{V}_h}
    \frac{ b_h(\boldsymbol{q}_h, v_h) }{\tnorm{ \boldsymbol{v}_h }_v},
\end{equation}
see~\cite[Lemma~4.8 and Eq.~102]{Rhebergen:2017} and~\cite[Lemma
1]{Rhebergen:2018b}, respectively. Next we note that the
velocity-pressure coupling term \cref{eq:formB} can be written as
\begin{equation}
  \label{eq:formBsplit}
  b_h(\boldsymbol{q}_h, v_h) := b_1(q_h, v_h) + b_2(\bar{q}_h, v_h),
\end{equation}
where
\begin{equation}
  \label{eq:bhbar}
  b_1(q_h,v_h) := - \del{q_h, \nabla \cdot v_h}_{\mathcal{T}}
  \quad \text{and} \quad
  b_2(\bar{q}_h, v_h) := \left\langle v_h \cdot n, \bar{q}_h \right\rangle_{\partial \mathcal{T}}.
\end{equation}
It follows immediately from \cref{eq:bound_bh} that
\begin{equation}
  \label{eq:b2bounded}
  \envert{b_2(\bar{q}_h, v_h)} \le c_b^b \tnorm{\boldsymbol{v}_h}_v\norm{\bar{q}_h}_p .
\end{equation}
Furthermore, in \cite[Lemma 3]{Rhebergen:2018b} we proved that there
exists a constant $\bar{\beta} > 0$, independent of $h$, such that for
all $\bar{q}_h\in\bar{Q}_h$
\begin{equation}
  \label{eq:stab_bhbar}
  \bar{\beta} \norm{\bar{q}_h}_{p} \le \sup_{\boldsymbol{v}_h\in \boldsymbol{V}_h}
  \frac{ b_2(\bar{q}_h, v_h) }{\tnorm{ \boldsymbol{v}_h }_v}.
\end{equation}
Stability of $b_2$ holds also for velocities that are divergence-free on
each element $K \in \mathcal{T}$, i.e.
\begin{equation}
  \label{eq:stab_bhbar_div0}
  \bar{\beta} \norm{\bar{q}_h}_{p} \le \sup_{\boldsymbol{v}_h\in \boldsymbol{V}_h^0}
  \frac{ b_2(\bar{q}_h, v_h) }{\tnorm{ \boldsymbol{v}_h }_v}
\end{equation}
where $\boldsymbol{V}_h^0 := V_h^0 \times \bar{V}_h$ and
\begin{equation}
  \label{eq:defVh0}
  \begin{split}
    V_h^0
    :&= \cbr{ v_h \in V_h\ :\ b_1(v_h, q_h) = 0\ \forall q_h \in Q_h}
    \\
    &= \cbr{ v_h \in V_h\ :\ (\nabla\cdot v_h)|_K = 0\ \forall K \in \mathcal{T}}.
  \end{split}
\end{equation}
\Cref{eq:stab_bhbar_div0} is a direct consequence of the inf-sup
condition in \cref{eq:bound_bh} and \cref{thm:howell}. This result will
play an important role in constructing a preconditioner in
\cref{sec:preconditioning}.

\subsection{The matrix formulation}
\label{ss:matrixformulation}

To express \cref{eq:discrete_problem} as a linear algebra problem, we
let $u \in \mathbb{R}^{n_u}$ and $\bar{u} \in \mathbb{R}^{\bar{n}_u}$ be
vectors of the discrete element and trace velocities
(degrees-of-freedom) with respect to the basis for $V_h$ and
$\bar{V}_h$, respectively.  Similarly, we let $p \in \mathbb{R}^{n_q}$
and $\bar{p} \in \mathbb{R}^{\bar{n}_q}$ be the degrees-of-freedom
associated with the basis for $Q_h$ and $\bar{Q}_h$, respectively. We
define also $\mathbb{V} := \cbr[0]{ {\bf v} = [v^T\, \bar{v}^T]^T\, : v
\in \mathbb{R}^{n_u},\ \bar{v} \in \mathbb{R}^{\bar{n}_u}}$ and
$\mathbb{Q} := \cbr[0]{ {\bf q} = [q^T\, \bar{q}^T]^T\, : q \in
\mathbb{R}^{n_q},\ \bar{q} \in \mathbb{R}^{\bar{n}_q}}$ and denote by
${\bf u} = [u^T \ \bar{u}^T]^T \in \mathbb{V}$ and ${\bf p} = [p^T \
\bar{p}^T]^T \in \mathbb{Q}$.

Next, let $A \in \mathbb{R}^{(n_u + \bar{n}_u) \times (n_u +
\bar{n}_u)}$ be the symmetric matrix defined by
\begin{equation}
  \label{eq:matrix_A}
  a_h(\boldsymbol{v}_h, \boldsymbol{v}_h) = {\bf v}^T A {\bf v}
  \qquad \text{where} \qquad
  A :=
  \begin{bmatrix}
    A_{uu} & A_{\bar{u}u}^T \\
    A_{\bar{u}u} & A_{\bar{u}\bar{u}}
  \end{bmatrix},
\end{equation}
for any ${\bf v} \in \mathbb{V}$. Here $A_{uu}$, $A_{\bar{u}u}$ and
$A_{\bar{u} \bar{u}}$ are the matrices obtained from the
discretization of $a_h((\cdot, 0), (\cdot, 0))$,
$a_h((\cdot, 0), (0, \cdot))$ and $a_h((0, \cdot), (0, \cdot))$,
respectively. Similarly, let
$B_1 \in \mathbb{R}^{n_q \times (n_u + \bar{n}_u)}$ and
$B_2 \in \mathbb{R}^{\bar{n}_q \times (n_u + \bar{n}_u)}$ be the
matrices defined by
\begin{align}
  b_1(q_h, v_h) &= q^TB_1{\bf v} \qquad \text{where} &
  B_1 &:=
  \begin{bmatrix}
    B_{pu} & 0
  \end{bmatrix},
  \\
  \label{eq:def_B2}
  b_2(\bar{q}_h, v_h) &= \bar{q}^TB_2{\bf v} \qquad \text{where} &
  B_2 &:=
  \begin{bmatrix}
    B_{\bar{p}u} & 0
  \end{bmatrix},
\end{align}
for any $q \in \mathbb{R}^{n_q}$ and $\bar{q} \in
\mathbb{R}^{\bar{n}_q}$. Here $B_{pu}$ and $B_{\bar{p}u}$ are the
matrices obtained from the discretization of $b_h((\cdot,0), \cdot)$ and
$b_h((0, \cdot), \cdot)$, respectively. Finally we define $L_u$ such
that
\begin{equation}
  (v_h, f)_{\mathcal{T}} = {\bf v}^T L
  \qquad \text{where} \qquad
  L :=
  \begin{bmatrix}
    L_u \\ 0
  \end{bmatrix}.
\end{equation}
We can now express \cref{eq:discrete_problem} in block matrix form as
\begin{equation}
  \label{eq:fullmatrix}
  \begin{bmatrix}
    A_{uu} & A_{\bar{u}u}^T & B_{pu}^T & B_{\bar{p}u}^T \\
    A_{\bar{u}u} & A_{\bar{u}\bar{u}} & 0 & 0 \\
    B_{pu} & 0 & 0 & 0 \\
    B_{\bar{p}u} & 0 & 0 & 0
  \end{bmatrix}
  \begin{bmatrix}
    u \\ \bar{u} \\ p \\ \bar{p}
  \end{bmatrix}
  =
  \begin{bmatrix}
    L_u \\
    0 \\ 0 \\ 0
  \end{bmatrix}.
\end{equation}

The matrices $A_{uu}$ and $B_{pu}$ in \cref{eq:fullmatrix} are block
diagonal (one block per cell), therefore $u$ and $p$ and can be
eliminated locally via static condensation. This leads to the
\emph{two-field reduced system}:
\begin{equation}
  \label{eq:condensed2form_up}
  \underbrace{
  \begin{bmatrix}
    \bar{A}^d &
    -A_{\bar{u}u}A_{uu}^{-1}\mathcal{P}^TB_{\bar{p}u}^T \\
    -B_{\bar{p}u}\mathcal{P}A_{uu}^{-1}A_{\bar{u}u}^T &
    -B_{\bar{p}u}\mathcal{P}A_{uu}^{-1}\mathcal{P}^TB_{\bar{p}u}^T
  \end{bmatrix}}_{\bar{\mathbb{A}}}
  \begin{bmatrix}
    \bar{u} \\ \bar{p}
  \end{bmatrix} =
  \begin{bmatrix}
     - A_{\bar{u}u}\mathcal{P}A_{uu}^{-1}L_u \\
    - B_{\bar{p}u}\mathcal{P}A_{uu}^{-1}L_u
  \end{bmatrix},
\end{equation}
with
\begin{equation}
  \label{eq:defAbard_up}
  \bar{A}^d := -A_{\bar{u}u}\mathcal{P}A_{uu}^{-1}A_{\bar{u}u}^T + A_{\bar{u}\bar{u}},
\end{equation}
and
\begin{equation}
  \label{eq:projection}
  \mathcal{P} := I - \Pi,
  \quad
  \Pi := A_{uu}^{-1}B_{pu}^T(B_{pu}A_{uu}^{-1}B_{pu}^T)^{-1}B_{pu}.
\end{equation}
Note that $\bar{A}^d$ is symmetric, and we remark that $\mathcal{P}$ is
an oblique projection matrix into the null-space of $B_{pu}$. It is
important to note that $\mathcal{P}$ can be assembled element-wise and
is therefore a local operator. Then, given the trace velocity $\bar{u}$
and the trace pressure $\bar{p}$, the element velocity $u$ and pressure
$p$ can be computed element-wise in a post-processing step by solving:
\begin{equation}
  \label{eq:localdiscreteProb_up}
  \begin{bmatrix}
    A_{uu} & B_{pu}^T \\
    B_{pu} & 0
  \end{bmatrix}
  \begin{bmatrix}
    u \\ p
  \end{bmatrix}
  =
  \begin{bmatrix}
    L_u - A_{\bar{u}u}^T\bar{u} - B_{\bar{p}u}^T\bar{p}
    \\
    0
  \end{bmatrix},
\end{equation}
resulting in
\begin{equation}
  \label{eq:localdiscreteSol_up}
  \begin{bmatrix}
    u \\ p
  \end{bmatrix}
  =
  \begin{bmatrix}
    \mathcal{P}A_{uu}^{-1}(L_u - A_{\bar{u}u}^T\bar{u} - B_{\bar{p}u}^T\bar{p}) \\
    (B_{pu}A_{uu}^{-1}B_{pu}^T)^{-1}B_{pu}A_{uu}^{-1}L_u
  \end{bmatrix}.
\end{equation}
It is clear that once $\bar{u}$ and $\bar{p}$ are known, $u =
\mathcal{P}A_{uu}^{-1}(L_u - A_{\bar{u}u}^T\bar{u} -
B_{\bar{p}u}^T\bar{p}) \in \Ker\, B_{pu}$ exactly (up to machine
precision) due to the application of $\mathcal{P}$ to the vector
$A_{uu}^{-1}(L_u - A_{\bar{u}u}^T\bar{u} - B_{\bar{p}u}^T\bar{p})$. In
other words, the discrete velocity $u_h$ is exactly divergence free on
each element.

\subsection{Matrix properties}

The following properties of the different matrices will be useful when
constructing a preconditioner for~\cref{eq:condensed2form_up} in
\cref{sec:preconditioning}.

\begin{lemma}
  \label{lem:mat2fieldInv}
  The matrix of the two-field reduced system \cref{eq:condensed2form_up}
  is invertible.
\end{lemma}
\begin{proof}
  Write \cref{eq:fullmatrix} as
  \begin{equation}
    \label{eq:pf_fm}
    \begin{bmatrix}
      A_{uu} & B_{pu}^T & A_{\bar{u}u}^T & B_{\bar{p}u}^T \\
      B_{pu} & 0 & 0 & 0 \\
      A_{\bar{u}u} & 0 & A_{\bar{u}\bar{u}} & 0 \\
      B_{\bar{p}u} & 0 & 0 & 0
    \end{bmatrix}
    \begin{bmatrix}
      u \\ p \\ \bar{u} \\ \bar{p}
    \end{bmatrix}
    =
    \begin{bmatrix}
      L_u \\ 0 \\ 0 \\ 0
    \end{bmatrix}.
  \end{equation}
  Define
  \begin{equation}
    \mathcal{A} =
    \begin{bmatrix}
      A_{uu} & B_{pu}^T \\
      B_{pu}& 0
    \end{bmatrix},
    \quad
    \mathcal{B} =
    \begin{bmatrix}
      A_{\bar{u}u} & 0 \\
      B_{\bar{p}u} & 0
    \end{bmatrix},
    \quad
    \mathcal{C} =
    \begin{bmatrix}
      A_{\bar{u}\bar{u}} & 0 \\
      0 & 0
    \end{bmatrix}.
  \end{equation}
  We can then write the matrix in \cref{eq:pf_fm} as the following block
  triangular factorization:
  \begin{equation}
    \mathbb{A}
    =
    \begin{bmatrix}
      \mathcal{A} & \mathcal{B}^T \\
      \mathcal{B} & \mathcal{C}
    \end{bmatrix}
    =
    \begin{bmatrix}
      I & 0 \\
      \mathcal{B}\mathcal{A}^{-1} & I
    \end{bmatrix}
    \begin{bmatrix}
      \mathcal{A} & 0 \\
      0 & \mathcal{S}
    \end{bmatrix}
    \begin{bmatrix}
      I & \mathcal{A}^{-1}\mathcal{B}^T \\
      0 & I
    \end{bmatrix},
  \end{equation}
  where the Schur complement $\mathcal{S} = \mathcal{C} -
  \mathcal{B}\mathcal{A}^{-1}\mathcal{B}^T$ is exactly the matrix of the
  two-field reduced system \cref{eq:condensed2form_up}. Since
  $\mathcal{A}$ is invertible ($A_{uu}$ is positive definite and
  $B_{pu}$ is full rank since it satisfies the inf-sup condition), we
  have that $\mathbb{A}$ is nonsingular if and only if $\mathcal{S}$ is
  nonsingular. Since $\mathbb{A}$ is nonsingular (due to well-posedness
  of \cref{eq:discrete_problem} \cite{Rhebergen:2017}), the result
  follows.
\end{proof}

Next, let us recall the following result from \cite[Chapter~A.5.5]{Boyd:book}.
\begin{proposition}
  \label{prop:pdpsd}
  Consider the following symmetric matrix:
  \begin{equation*}
    M =
    \begin{bmatrix}
      A & B^T \\ B & C
    \end{bmatrix}.
  \end{equation*}
  $M$ is positive-definite $\Leftrightarrow$ $A$ and $C-BA^{-1}B^T$ are
  positive-definite.
\end{proposition}

This proposition is used now to prove the following lemma.

\begin{lemma}
  \label{lem:AbarSPD}
  The matrix
  $\bar{A} = A_{\bar{u}\bar{u}} -
  A_{\bar{u}u}A_{uu}^{-1}A_{\bar{u}u}^T$ is symmetric-positive
  definite.
\end{lemma}
\begin{proof}
  It is clear that $\bar{A}$ is symmetric. By \cref{prop:pdpsd} we
  know that $A$ in \cref{eq:matrix_A} is positive definite if and only
  if $A_{uu}$ and $\bar{A}$ are positive-definite. We know that $A$ is
  symmetric positive-definite by \cref{eq:stab_bound_ah} so that
  $A_{uu}$ and $\bar{A}$ are symmetric positive-definite.
\end{proof}

\begin{lemma}
  \label{lem:barAdspdmat}
  The matrix $\bar{A}^d$ in \cref{eq:defAbard_up} is symmetric
  positive-definite.
\end{lemma}
\begin{proof}
  The symmetry of $\bar{A}^d$ is clear. Using the definition of
  $\mathcal{P}$ in \cref{eq:projection} we have:
  \begin{equation}
    \bar{A}^d = \bar{A} + A_{\bar{u}u}\Pi A_{uu}^{-1}A_{\bar{u}u}^T,
  \end{equation}
  where
  $\bar{A} = -A_{\bar{u}u}A_{uu}^{-1}A_{\bar{u}u}^T +
  A_{\bar{u}\bar{u}}$. By \cref{lem:AbarSPD} we know that
  $\bar{A} = -A_{\bar{u}u}A_{uu}^{-1}A_{\bar{u}u}^T +
  A_{\bar{u}\bar{u}}$ is symmetric positive-definite. Consider now
  $A_{\bar{u}u}\Pi A_{uu}^{-1}A_{\bar{u}u}^T$. Let
  $S_{pp} = B_{pu}A_{uu}^{-1}B_{pu}^T$ and note that $S_{pp}$ is
  symmetric and positive-definite (since $A_{uu}$ is symmetric
  positive-definite and $B_{pu}$ is full rank). Then
  \begin{equation}
    \begin{split}
      \langle A_{\bar{u}u}\Pi A_{uu}^{-1}A_{\bar{u}u}^T\bar{x}, \bar{x} \rangle
      &=
      \langle A_{\bar{u}u} A_{uu}^{-1}B_{pu}^TS_{pp}^{-1}B_{pu}A_{uu}^{-1}A_{\bar{u}u}^T\bar{x}, \bar{x} \rangle
      \\
      &=
      \langle S_{pp}^{-1}B_{pu}A_{uu}^{-1}A_{\bar{u}u}^T\bar{x}, B_{pu}A_{uu}^{-1}A_{\bar{u}u}^T\bar{x} \rangle
      \\
      &=
      \langle B_{pu}A_{uu}^{-1}A_{\bar{u}u}^T\bar{x}, B_{pu}A_{uu}^{-1}A_{\bar{u}u}^T\bar{x} \rangle_{S_{pp}^{-1}}
      \\
      &\ge 0,
    \end{split}
  \end{equation}
  hence $A_{\bar{u}u}\Pi A_{uu}^{-1}A_{\bar{u}u}^T$ is symmetric
  positive semidefinite and $\bar{A}^d$ must be positive-definite.
\end{proof}

\begin{lemma}
  \label{lem:schurcomp2fieldinvert}
  The Schur complement of the matrix of the two-field reduced system
  \cref{eq:condensed2form_up} is invertible.
\end{lemma}
\begin{proof}
  The proof is the same as that of \cref{lem:mat2fieldInv} but with
  \begin{equation}
    \mathcal{A} = \bar{A}^d, \quad
    \mathcal{B} = -B_{\bar{p}u}\mathcal{P}A_{uu}^{-1}A_{\bar{u}u}^T, \quad
    \mathcal{C} = -B_{\bar{p}u}\mathcal{P}A_{uu}^{-1}B_{\bar{p}u}^T.
  \end{equation}
  By \cref{lem:mat2fieldInv} we know that the matrix of the two-field
  reduced system \cref{eq:condensed2form_up} is invertible and by
  \cref{lem:barAdspdmat} we know that $\mathcal{A}$ is symmetric
  positive-definite. The result follows.
\end{proof}

\section{Preconditioning}
\label{sec:preconditioning}

We present now a provably optimal preconditioner for the condensed
problem in~\cref{eq:condensed2form_up}.

\subsection{Block preconditioner}
\label{ss:elim_vel_pres}

We first introduce some definitions. The (negative) trace pressure Schur
complement of the matrix in \cref{eq:condensed2form_up} is given by
\begin{equation}
  \label{eq:facetpres_schurcomp}
  \bar{S}
  := B_{\bar{p}u}\mathcal{P}\del[1]{ A_{uu}^{-1} + A_{uu}^{-1}A_{\bar{u}u}^T (\bar{A}^d)^{-1}A_{\bar{u}u}A_{uu}^{-1}}
        \mathcal{P}^TB_{\bar{p}u}^T.
\end{equation}
The element $M$ and trace $\bar{M}$ pressure mass matrices are defined
by, respectively,
\begin{equation}
  \label{eq:def_cell_facet_mmatrix}
  \norm{q_h}_{\Omega}^2 = q^TMq, \qquad
  \norm{\bar{q}_h}_p^2 = \bar{q}^T\bar{M}\bar{q}.
\end{equation}
We will also require the following matrix
\begin{equation}
  \label{eq:AP_matrix_def}
  A_{\mathcal{P}} =
  \begin{bmatrix}
    A_{uu} & \mathcal{P}^TA_{\bar{u}u}^T \\
    A_{\bar{u}u}\mathcal{P} & A_{\bar{u}\bar{u}}
  \end{bmatrix}.
\end{equation}
Next, let $\mathbb{V}^0 := \cbr[0]{ {\bf v} \in \mathbb{V} : v \in
\Ker\, B_{pu}}$. It is then easy to show, using that $\mathcal{P}v
= v$ for ${\bf v} \in \mathbb{V}^0$, that
\begin{equation}
  \label{eq:innerprodA_Ap}
  \langle A{\bf v}, {\bf v} \rangle = \langle A_{\mathcal{P}}{\bf v}, {\bf v} \rangle
  \quad \forall {\bf v} \in \mathbb{V}^0.
\end{equation}

\begin{lemma}
  The matrix $A_{\mathcal{P}}$ in \cref{eq:AP_matrix_def} is symmetric
  positive-definite.
\end{lemma}
\begin{proof}
  We first note that $A_{uu}$ is symmetric and positive-definite.
  Second, we note that the Schur complement of $A_{\mathcal{P}}$ is
  $\bar{A}^d$ which, by \cref{lem:barAdspdmat} is symmetric
  positive-definite. The result follows by \cref{prop:pdpsd}.
\end{proof}

In the following two lemmas we show that $\bar{S}$ is spectrally
equivalent to $\bar{M}$ and $B_{\bar{p}u}A_{uu}^{-1}B_{\bar{p}u}^T$.

\begin{lemma}
  \label{lem:spec_equiv_Mbar_Sbar}
  Let $\bar{S}$ be the matrix defined in \cref{eq:facetpres_schurcomp}
  and let $\bar{M}$ be the trace pressure mass matrix defined in
  \cref{eq:def_cell_facet_mmatrix}. Let $c_a^s$ and $c_a^b$ be the
  constants given in \cref{eq:stab_bound_ah} and let $c_b^b$ and
  $\bar{\beta}$ be the constants given in \cref{eq:b2bounded} and
  \cref{eq:stab_bhbar}, respectively. The following holds:
  \begin{equation}
    \label{eq:spectralequivalence}
    \frac{\bar{\beta}^2}{c_a^b}
    \le
    \frac{\bar{q}^T\bar{S}\bar{q}}{\bar{q}^T\bar{M}\bar{q}}
    \le
    \frac{(c_b^b)^2}{c_a^s}.
  \end{equation}
\end{lemma}
\begin{proof}
  Stability of $b_2$ (see \cref{eq:stab_bhbar_div0}) and equivalence of
  $a_h$ with $\tnorm{\cdot}_v$ in $\boldsymbol{V}_h$ (see
  \cref{eq:stab_bound_ah}) imply
  \begin{equation}
    \label{eq:stab_bhbar_div0_proj}
    \frac{\bar{\beta}}{\sqrt{c_a^b}}
    \le \min_{\substack{\bar{q}_h \in \mathbb{R}^{\bar{n}_q}
    \\
    \bar{q}_h \ne 0}}\max_{\boldsymbol{v}_h\in \boldsymbol{V}_h^0}\frac{ b_2(\bar{q}_h, v_h) }
    {a_h(\boldsymbol{v}_h, \boldsymbol{v}_h)^{1/2}\norm{\bar{q}_h}_{p}}.
  \end{equation}
  We may write \cref{eq:stab_bhbar_div0_proj} in matrix notation as:
  \begin{equation}
    \label{eq:minmaxmatrixnot}
    \frac{\bar{\beta}}{\sqrt{c_a^b}}
    \le \min_{\substack{\bar{q}_h \in \mathbb{R}^{\bar{n}_q}
    \\
    \bar{q}_h \ne 0}}\max_{{\bf v} \in \mathbb{V}^0}\frac{ \langle \bar{q}, B_2{\bf v} \rangle }
    {\langle A{\bf v}, {\bf v} \rangle^{1/2} \langle \bar{M}\bar{q}, \bar{q} \rangle^{1/2}}.
  \end{equation}
  Let $\mathcal{R}$ be a $2 \times 2$ block diagonal matrix with in the
  top left block $\mathcal{P}$ and the bottom right block the identity
  matrix $\bar{I}_u \in \mathbb{R}^{\bar{n}_u\times\bar{n}_u}$. Using
  \cref{eq:innerprodA_Ap} and the property
  $\langle \bar{q}, B_2{\bf v} \rangle = \langle
  \mathcal{R}^TB_2^T\bar{q}, {\bf v} \rangle$ for
  ${\bf v} \in \mathbb{V}^0$, following similar steps as in
  \cite[Chapter~3]{Elman:book},
  \begin{equation}
    \begin{split}
      \frac{\bar{\beta}}{\sqrt{c_a^b}}
      &\le \min_{\substack{\bar{q}_h \in \mathbb{R}^{\bar{n}_q} \\\bar{q}_h \ne 0}} \frac{1}{\langle \bar{M}\bar{q}, \bar{q} \rangle^{1/2}}
      \max_{{\bf v} \in \mathbb{V}^0}\frac{ \langle \mathcal{R}^TB_2^T\bar{q}, {\bf v} \rangle }
      {\langle A_{\mathcal{P}}{\bf v}, {\bf v} \rangle^{1/2} }
      \\
      &\le \min_{\substack{\bar{q}_h \in \mathbb{R}^{\bar{n}_q} \\\bar{q}_h \ne 0}} \frac{1}{\langle \bar{M}\bar{q}, \bar{q} \rangle^{1/2}}
      \max_{{\bf v} \in \mathbb{V}}\frac{ \langle \mathcal{R}^TB_2^T\bar{q}, {\bf v} \rangle }
      {\langle A_{\mathcal{P}}{\bf v}, {\bf v} \rangle^{1/2} }
      \\
      &= \min_{\substack{\bar{q}_h \in \mathbb{R}^{\bar{n}_q} \\\bar{q}_h \ne 0}} \frac{1}{\langle \bar{M}\bar{q}, \bar{q} \rangle^{1/2}}
      \max_{{\bf v} \in \mathbb{V}}\frac{ \langle \mathcal{R}^TB_2^T\bar{q}, A_{\mathcal{P}}^{-1/2}A_{\mathcal{P}}^{1/2}{\bf v} \rangle }
      {\langle A_{\mathcal{P}}^{1/2}{\bf v}, A_{\mathcal{P}}^{1/2}{\bf v} \rangle^{1/2} }
      \\
      &= \min_{\substack{\bar{q}_h \in \mathbb{R}^{\bar{n}_q} \\\bar{q}_h \ne 0}} \frac{1}{\langle \bar{M}\bar{q}, \bar{q} \rangle^{1/2}}
      \max_{{\bf w}=A_{\mathcal{P}}^{1/2}{\bf v}, \ {\bf v} \in \mathbb{V}}\frac{ \langle \mathcal{R}^TB_2^T\bar{q}, A_{\mathcal{P}}^{-1/2}{\bf w} \rangle }
      {\langle {\bf w}, {\bf w} \rangle^{1/2} }
      \\
      &= \min_{\substack{\bar{q}_h \in \mathbb{R}^{\bar{n}_q} \\\bar{q}_h \ne 0}} \frac{1}{\langle \bar{M}\bar{q}, \bar{q} \rangle^{1/2}}
      \max_{{\bf w} \ne 0}\frac{ \langle A_{\mathcal{P}}^{-1/2}\mathcal{R}^TB_2^T\bar{q}, {\bf w} \rangle }
      {\langle {\bf w}, {\bf w} \rangle^{1/2} }.
    \end{split}
  \end{equation}
  For a given $\bar{q}$, the maximum is reached for ${\bf w} =
  A_{\mathcal{P}}^{-1/2}\mathcal{R}^TB_2^T\bar{q}$, hence
  \begin{equation}
    \frac{\bar{\beta}}{\sqrt{c_a^b}}
    \le \min_{\substack{\bar{q}_h \in \mathbb{R}^{\bar{n}_q} \\\bar{q}_h \ne 0}}
    \frac{ \langle B_2 \mathcal{R} A_{\mathcal{P}}^{-1}\mathcal{R}^TB_2^T\bar{q}, \bar{q} \rangle^{1/2} }
    {\langle \bar{M}\bar{q}, \bar{q} \rangle^{1/2}}.
  \end{equation}
  By direct computation, and using that $\mathcal{P}^2 = \mathcal{P}$
  and $\mathcal{P}A_{uu}^{-1} = A_{uu}^{-1}\mathcal{P}^T$, we note that
  $B_2 \mathcal{R} A_{\mathcal{P}}^{-1}\mathcal{R}^TB_2^T = B_{\bar{p}u}
  \mathcal{P} (A_{\mathcal{P}}^{-1})_{11} \mathcal{P}^T B_{\bar{p}u}^T =
  \bar{S}$, proving the lower bound in \cref{eq:spectralequivalence}.

  For the upper bound, from stability and \cref{eq:stab_bound_ah} and
  boundedness \cref{eq:b2bounded} of $a_{h}$ on $\boldsymbol{V}_h$,
  \begin{equation}
    \label{eq:upperboundb2}
    b_2(\bar{q}_h, v_h)
    \le c_b^b \tnorm{\boldsymbol{v}_h}_v\norm{\bar{q}_h}_p
    \le \frac{c_b^b}{\sqrt{c_a^s}}a_h(\boldsymbol{v}_h, \boldsymbol{v}_h)^{1/2}\norm{\bar{q}_h}_p.
  \end{equation}
  In matrix form this reads as
  \begin{equation}
    \langle \bar{q}, B_2{\bf v} \rangle
    \le \frac{c_b^b}{\sqrt{c_a^s}} \langle A{\bf v}, {\bf v} \rangle^{1/2}\langle \bar{M}\bar{q}, \bar{q} \rangle^{1/2},
    \qquad \forall {\bf v} \in \mathbb{V}.
  \end{equation}
  Set ${\bf v} = \mathcal{R}{\bf w}$ for ${\bf w} \in \mathbb{V}$. Then,
  using \cref{eq:innerprodA_Ap} and since $\mathcal{R}{\bf w} \in
  \mathbb{V}^0$, we find
  \begin{equation}
    \label{eq:qB2wubAp}
      \langle \bar{q}, B_2\mathcal{R}{\bf w} \rangle
      \le
      \frac{c_b^b}{\sqrt{c_a^s}} \langle A\mathcal{R}{\bf w}, \mathcal{R}{\bf w} \rangle^{1/2}
      \langle \bar{M}\bar{q}, \bar{q} \rangle^{1/2}
      =
      \frac{c_b^b}{\sqrt{c_a^s}} \langle A_{\mathcal{P}}\mathcal{R}{\bf w}, \mathcal{R}{\bf w} \rangle^{1/2}
      \langle \bar{M}\bar{q}, \bar{q} \rangle^{1/2},
  \end{equation}
  which holds for all ${\bf w} \in \mathbb{V}$. We next show that
  $\langle A_{\mathcal{P}}\mathcal{R}{\bf w}, \mathcal{R}{\bf w}
  \rangle \le \langle A_{\mathcal{P}}{\bf w}, {\bf w} \rangle$ for all
  ${\bf w} \in \mathbb{V}$. By definition and using that
  $\mathcal{P}^T = A_{uu}\mathcal{P}A_{uu}^{-1}$,
  \begin{equation}
    \label{eq:ApRubAp}
    \begin{split}
      \langle A_{\mathcal{P}}\mathcal{R}{\bf w}, \mathcal{R}{\bf w} \rangle
      &=
      \begin{bmatrix}
        w^T\mathcal{P}^T & \bar{w}^T
      \end{bmatrix}
      \begin{bmatrix}
        A_{uu} & \mathcal{P}^TA_{\bar{u}u}^T \\
        A_{\bar{u}u}\mathcal{P} & A_{\bar{u}\bar{u}}
      \end{bmatrix}
      \begin{bmatrix}
        \mathcal{P}w \\ \bar{w}
      \end{bmatrix}
      \\
      &=w^T\mathcal{P}^TA_{uu}\mathcal{P}w + \bar{w}^TA_{\bar{u}u}\mathcal{P}w
      + (\mathcal{P}w)^TA_{\bar{u}u}\bar{w} + \bar{w}^TA_{\bar{u}\bar{u}}\bar{w}
      \\
      &=w^TA_{uu}\mathcal{P}w + \bar{w}^TA_{\bar{u}u}\mathcal{P}w
      + (\mathcal{P}w)^TA_{\bar{u}u}\bar{w} + \bar{w}^TA_{\bar{u}\bar{u}}\bar{w}
      \\
      &=w^TA_{uu}(I-\Pi)w + \bar{w}^TA_{\bar{u}u}\mathcal{P}w
      + (\mathcal{P}w)^TA_{\bar{u}u}\bar{w} + \bar{w}^TA_{\bar{u}\bar{u}}\bar{w}
      \\
      &=\langle A_{\mathcal{P}}{\bf w}, {\bf w} \rangle - \langle A_{uu}\Pi w, w \rangle
      \\
      &=\langle A_{\mathcal{P}}{\bf w}, {\bf w} \rangle - \langle B_{pu}^TS^{-1}_{pp}B_{pu} w, w \rangle
      \\
      &=\langle A_{\mathcal{P}}{\bf w}, {\bf w} \rangle - \langle B_{pu} w, B_{pu}w \rangle_{S^{-1}_{pp}}
      \\
      &\le \langle A_{\mathcal{P}}{\bf w}, {\bf w} \rangle,
    \end{split}
  \end{equation}
  where we used that $S_{pp} = B_{pu}A_{uu}^{-1}B_{pu}^T$ is symmetric
  and positive-definite. Combining \cref{eq:qB2wubAp} and
  \cref{eq:ApRubAp} we obtain
  \begin{equation}
    \langle \bar{q}, B_2\mathcal{R}{\bf w} \rangle
    \le
    \frac{c_b^b}{\sqrt{c_a^s}} \langle A_{\mathcal{P}}{\bf w}, {\bf w} \rangle^{1/2}
    \langle \bar{M}\bar{q}, \bar{q} \rangle^{1/2} \qquad \forall {\bf w} \in \mathbb{V}.
  \end{equation}
  The result follows after dividing both sides by $\langle
  A_{\mathcal{P}}{\bf w}, {\bf w} \rangle^{1/2}\langle \bar{M}\bar{q},
  \bar{q} \rangle^{1/2}$ and following similar steps as used to obtain
  the lower bound.
\end{proof}

\begin{lemma}
  \label{lem:spec_equiv_Sbar_BpuABpuT}
  Let $\bar{S}$ be the matrix defined in \cref{eq:facetpres_schurcomp}
  and let $A_{uu}$ and $B_{\bar{p}u}$ be matrices defined in,
  respectively, \cref{eq:matrix_A} and \cref{eq:def_B2}. Then
  $\bar{S}$ is spectrally equivalent to
  $B_{\bar{p}u}A_{uu}^{-1}B_{\bar{p}u}^T$.
\end{lemma}
\begin{proof}
  We first remark that a result of \cite[Theorem~3]{Rhebergen:2018b} is
  that $B_{\bar{p}u}A_{uu}^{-1}B_{\bar{p}u}^T$ is spectrally equivalent
  to the trace pressure mass matrix $\bar{M}$ defined in
  \cref{eq:def_cell_facet_mmatrix}. Since $\bar{M}$ is spectrally
  equivalent to $\bar{S}$ by \cref{lem:spec_equiv_Mbar_Sbar}, the result
  follows.
\end{proof}

We now formulate a preconditioner for the statically condensed linear
system in \cref{eq:condensed2form_up} using
\Cref{lem:spec_equiv_Mbar_Sbar,lem:spec_equiv_Sbar_BpuABpuT}.

\begin{theorem}[A preconditioner for the statically condensed discrete
  Stokes problem]
  \label{thm:P2x2precon}
  Let $\bar{\mathbb{A}}$ be the matrix given in
  \cref{eq:condensed2form_up}, let $\bar{M}$ be the trace pressure
  mass matrix defined in \cref{eq:def_cell_facet_mmatrix}, let
  $A_{uu}$ and $B_{\bar{p}u}$ be matrices defined in, respectively,
  \cref{eq:matrix_A} and \cref{eq:def_B2}, and let $\bar{A}^d$ be the
  matrix defined in \cref{eq:defAbard_up}. Let $\bar{R}^d$ be an
  operator that is spectrally equivalent to $\bar{A}^d$ and define
  \begin{equation}
    \label{eq:P22}
    \mathbb{P}_{\bar{M}}^{-1} =
    \begin{bmatrix}
      (\bar{R}^d)^{-1} & 0 \\ 0 & \bar{M}^{-1}
    \end{bmatrix},
    \qquad
    \mathbb{P}_{BAB}^{-1} =
    \begin{bmatrix}
      (\bar{R}^d)^{-1} & 0 \\ 0 & (B_{\bar{p}u}A_{uu}^{-1}B_{\bar{p}u}^T)^{-1}
    \end{bmatrix}.
  \end{equation}
  There exist positive constants $C_1$, $C_2$, $C_3$, $C_4$,
  independent of $h$, such that eigenvalues of both
  $\mathbb{P}_{\bar{M}}^{-1}\bar{\mathbb{A}}$ and
  $\mathbb{P}_{BAB}^{-1}\bar{\mathbb{A}}$ satisfy
  $\lambda \in [-C_1, -C_2] \cup [C_3, C_4]$.
\end{theorem}
\begin{proof}
  By \cref{lem:spec_equiv_Mbar_Sbar,lem:spec_equiv_Sbar_BpuABpuT} we
  know that $\bar{M}$ and $B_{\bar{p}u}A_{uu}^{-1}B_{\bar{p}u}^T$ are
  spectrally equivalent to the pressure Schur complement $\bar{S}$ of
  $\bar{\mathbb{A}}$. Furthermore, $\bar{S}$ is invertible (see
  \cref{lem:schurcomp2fieldinvert}) and $\bar{A}^d$ is symmetric
  positive-definite (see \cref{lem:barAdspdmat}). The result then
  follows by direct application of \cite[Theorem~5.2]{Pestana:2015}.
\end{proof}

A consequence of \cref{thm:P2x2precon} is that the number of MINRES
iterations needed to solve \cref{eq:condensed2form_up} when
preconditioned by $\mathbb{P}_{\bar{M}}^{-1}$ or by
$\mathbb{P}_{BAB}^{-1}$ to a given tolerance will be independent of the
size of the discrete problem. In other words,
$\mathbb{P}_{\bar{M}}^{-1}$ and $\mathbb{P}_{BAB}^{-1}$ are optimal
preconditioners for \cref{eq:condensed2form_up}.

\subsection{Choosing $(\bar{R}^d)^{-1}$}
\label{ss:char_barRd}

The preconditioners in \cref{eq:P22} require an operator $\bar{R}^d$
that is spectrally equivalent to $\bar{A}^d$. For notational purposes,
let us denote by $(\cdot)^{MG}$ the application of multigrid to
approximate the inverse of $(\cdot)$. In this section we then discuss
the following three choices: $(\bar{R}^d)^{-1} = (\bar{A}^d)^{-1}$,
$(\bar{R}^d)^{-1} = (\bar{A}^d)^{MG}$, and
$(\bar{R}^d)^{-1} = (\bar{A}_{\gamma})^{MG}$ where $\bar{A}_{\gamma}$
will be defined to approximate $\bar{A}^d$.

The optimal choice for $(\bar{R}^d)^{-1}$ is $(\bar{A}^d)^{-1}$.
However, computing the inverse (action) of $\bar{A}^d$ is too costly
in practice. The usual choice for Stokes solvers is to replace the
inverse of the discrete vector Laplacian by a V-cycle multigrid
method. However, as we show in the following two lemmas, standard
multigrid methods designed for $H^1$-like operators are not guaranteed
to perform well on~$\bar{A}^d$.

In what follows, we require the following definition of an orthogonal
subspace $Y^{\perp}$ of a linear subspace $Y$ of $\mathbb{R}^n$:
$Y^{\perp} := \cbr[0]{x \in \mathbb{R}^n\,:\, x^Ty=0\ \forall y \in
  Y}$.

\begin{lemma}
  \label{lem:KerPAinvImBpuT}
  Let $\mathcal{P}$ be given by \cref{eq:projection}. Then
  \begin{subequations}
    \begin{align}
      \label{eq:KerPtPit_a}
      \Ker\, \mathcal{P}^T &= \Image\, B_{pu}^T,
      \\
      \label{eq:KerPtPit_b}
      \Ker\, \Pi^T &= \Ker\, (B_{pu}A_{uu}^{-1}).
    \end{align}
  \end{subequations}
\end{lemma}
\begin{proof}
  We first show \cref{eq:KerPtPit_a}. We know that $\mathcal{P}$ is an
  oblique projector onto
  $\Ker\, B_{pu} = (\Image\, B_{pu}^T)^{\perp}$ (the equality
  is by \cite[Theorem 3.1.1]{Boffi:book}), i.e.,
  $\Image\, \mathcal{P} = (\Image\, B_{pu}^T)^{\perp}$. Then,
  again using \cite[Theorem 3.1.1]{Boffi:book},
  \begin{equation}
    \Ker\, \mathcal{P}^T
    = (\Image\, \mathcal{P})^{\perp}
    = \Image\, B_{pu}^T.
  \end{equation}

  The proof for \cref{eq:KerPtPit_b} is similar and is given here for
  completeness. First, note that
  $\Image\, \Pi = \Image\, (A_{uu}^{-1}B_{pu}^T) = (\Ker\,
  (B_{pu}A_{uu}^{-1}))^{\perp}$ (see, e.g., \cite[Page
  19]{Benzi:2005}). Then, using \cite[Theorem 3.1.1]{Boffi:book},
  \begin{equation}
    \Ker\, \Pi^T = (\Image\, \Pi)^{\perp} = \Ker\,
    (B_{pu}A_{uu}^{-1}).
  \end{equation}
\end{proof}

\begin{lemma}
  \label{lem:properties-barAd}
  For $\bar{x} \in \mathbb{R}^{\bar{n}_u}$ such that
  $A_{\bar{u}u}^T\bar{x} \in \Ker\, (B_{pu}A_{uu}^{-1})$ the
  following holds:
  \begin{equation}
    \label{eq:propertyAd_1}
    c_1\tnorm{\bar{x}_h}_h^2 \le
    \langle \bar{A}^d\bar{x}, \bar{x} \rangle
    = \langle \bar{A}\bar{x}, \bar{x} \rangle
    \le
    c_2\tnorm{\bar{x}_h}_h^2.
  \end{equation}
  If $\bar{x} \in \mathbb{R}^{\bar{n}_u}$ is such that
  $A_{\bar{u}u}^T\bar{x} \in \Image\, B_{pu}^T = (\Ker\,
  B_{pu})^{\perp}$, then
  \begin{equation}
    \label{eq:propertyAd_2}
    c_1\tnorm{\bar{x}_h}_h^2
    \le
    \langle \bar{A}^d\bar{x}, \bar{x} \rangle
    \le c_2 h^{-2}\tnorm{\bar{x}_h}_h^2.
  \end{equation}
\end{lemma}
\begin{proof}
  We first note that the lower bound is always true because (see the
  proof of \cref{lem:barAdspdmat})
  \begin{equation}
    c_1\tnorm{\bar{x}_h}_h^2 \le
    \langle \bar{A}\bar{x}, \bar{x} \rangle
    \le
    \langle \bar{A}\bar{x}, \bar{x} \rangle
    + \langle A_{\bar{u}u}\Pi A_{uu}^{-1}A_{\bar{u}u}^T\bar{x}, \bar{x} \rangle
    =
    \langle \bar{A}^d\bar{x}, \bar{x} \rangle \qquad \forall \bar{x} \in \mathbb{R}^{\bar{n}_u},
  \end{equation}
  where the first inequality is by \cite[Lemma 5]{Rhebergen:2018b}.

  To prove the upper bound in \cref{eq:propertyAd_1} we note that if
  $\bar{x}$ is such that
  $A_{\bar{u}u}^T\bar{x} \in \Ker\, (B_{pu}A_{uu}^{-1})$ then by
  \cref{lem:KerPAinvImBpuT}
  \begin{equation}
    \langle A_{\bar{u}u}\Pi A_{uu}^{-1}A_{\bar{u}u}^T\bar{x}, \bar{x} \rangle
    =
    \langle A_{\bar{u}u}A_{uu}^{-1} \Pi^T A_{\bar{u}u}^T\bar{x}, \bar{x} \rangle
    = 0,
  \end{equation}
  where we used that $\Pi A_{uu}^{-1} = A_{uu}^{-1} \Pi^T$. It follows
  that
  \begin{equation}
    \begin{split}
      \langle \bar{A}^d\bar{x}, \bar{x} \rangle
      &= \langle (A_{\bar{u}\bar{u}} -A_{\bar{u}u}\mathcal{P}A_{uu}^{-1}A_{\bar{u}u}^T)\bar{x}, \bar{x} \rangle
      = \langle (A_{\bar{u}\bar{u}} -A_{\bar{u}u}(I - \Pi)A_{uu}^{-1}A_{\bar{u}u}^T)\bar{x}, \bar{x} \rangle
      \\
      &= \langle (A_{\bar{u}\bar{u}} -A_{\bar{u}u}A_{uu}^{-1}A_{\bar{u}u}^T)\bar{x}, \bar{x} \rangle
      = \langle \bar{A}\bar{x}, \bar{x} \rangle.
    \end{split}
  \end{equation}
  The result \cref{eq:propertyAd_1} then follows from \cite[Lemma
  5]{Rhebergen:2018b}.

  Next we prove the upper bound in \cref{eq:propertyAd_2}. If
  $\bar{x}$ is such that
  $A_{\bar{u}u}^T\bar{x} \in \Image\, B_{pu}^T$, then
  $\mathcal{P}A_{uu}^{-1}A_{\bar{u}u}^T\bar{x} =
  \mathcal{P}A_{uu}^{-1}\mathcal{P}^TA_{\bar{u}u}^T\bar{x} = 0$ by
  \cref{lem:KerPAinvImBpuT}. Therefore,
  \begin{equation}
    \langle \bar{A}^d\bar{x}, \bar{x} \rangle = \langle A_{\bar{u}\bar{u}}\bar{x}, \bar{x} \rangle.
  \end{equation}
  It follows that
  \begin{equation}
    \begin{split}
      \langle A_{\bar{u}\bar{u}}\bar{x}, \bar{x} \rangle
      &= a_h((0, \bar{x}_h), (0, \bar{x}_h))
      \\
      &= \alpha \sum_{K \in \partial T} h_K^{-1} \int_{\partial K}\bar{x}_h \cdot \bar{x}_h \dif s
      \\
      &= \alpha \sum_{K \in \partial T} h_K^{-2} \del{ h_K\int_{\partial K}\bar{x}_h \cdot \bar{x}_h \dif s}
      \\
      &\le c h^{-2}\norm{\bar{x}_h}_{h,0}^2.
    \end{split}
  \end{equation}
  The result follows after applying \cref{eq:Poincare_facet}.
\end{proof}

\Cref{lem:properties-barAd} implies that if $\bar{x} \in
\mathbb{R}^{\bar{n}_u}$ is such that $A_{\bar{u}u}^T\bar{x} \in
\Ker\, (B_{pu}A_{uu}^{-1})$, then standard multigrid designed for
$H^1$-like operators will be a good preconditioner for $\bar{A}^d$.
However, if $\bar{x} \in \mathbb{R}^{\bar{n}_u}$ is such that
$A_{\bar{u}u}^T\bar{x} \in \Image\, B_{pu}^T$, we cannot conclude
anything about multigrid convergence. In fact, in numerical examples
(see \cref{sec:num_examples}) we see deterioration in performance of the
solver when the mesh is refined if we choose $(\bar{R}^d)^{-1} =
(\bar{A}^d)^{MG}$. The upper bound in \cref{eq:propertyAd_2} explains
this deterioration. The main conclusion we can draw from
\cref{lem:properties-barAd} then is that we should not choose
$(\bar{R}^d)^{-1} = (\bar{A}^d)^{MG}$ with multigrid methods designed
for $H^1$-like operators. In what follows we introduce a new
approximation to~$\bar{A}^d$.

We propose to set $(\bar{R}^d)^{-1} = (\bar{A}_{\gamma})^{MG}$ where
\begin{equation}
  \label{eq:def_Ag}
  \bar{A}_{\gamma} = -A_{\bar{u}u}\widehat{A}_{\gamma}^{-1}A_{\bar{u}u}^T + A_{\bar{u}\bar{u}},
\end{equation}
and where,
\begin{equation}
  \label{eq:Am_definition}
  \begin{split}
    \widehat{A}_{\gamma}^{-1}
    &=
    (A_{uu} + \gamma B_{pu}^TM^{-1}B_{pu})^{-1}
    \\
    &=
    A_{uu}^{-1} - A_{uu}^{-1}B_{pu}^T(\gamma^{-1}M + B_{pu}A_{uu}^{-1}B_{pu}^T)^{-1}B_{pu}A_{uu}^{-1},
  \end{split}
\end{equation}
where the second equality is by the Woodbury matrix identity which holds
for $0< \gamma < \infty$. The $\gamma B_{pu}^TM^{-1}B_{pu}$ term is
often added to $A_{uu}$ in the discretization of the linearized
Navier--Stokes equations to construct augmented Lagrangian
preconditioners (see, for example, \cite{Benzi:2011}). The $\gamma
B_{pu}^TM^{-1}B_{pu}$ term is also similar to grad-div stabilization and
can be added to discretizations of incompressible flows to improve mass
conservation in the case that $u \in \Ker\, B_{pu}$ does not
exactly imply $\nabla \cdot u_h = 0$. For example, it was shown in
\cite{Linke:2011} that the solution to the grad-div stabilized
Taylor--Hood discretization of the Navier--Stokes equations converges to
the Scott--Vogelius solution as $\gamma \to \infty$. Here we note that
$\bar{A}_{\gamma} \to \bar{A}^d$ when $\gamma \to \infty$. We now show,
for finite $\gamma$, that $\bar{A}_{\gamma}$ is equivalent to an
$H^1$-operator.

\begin{lemma}
  \label{lem:Agamma-H1equiv}
  Let $\bar{A}_{\gamma}$ be as defined in \cref{eq:def_Ag} and let
  $0 < \gamma < \infty$. Then $\bar{A}_{\gamma}$ is equivalent to an
  $H^1$-operator.
\end{lemma}
\begin{proof}
  Consider the problem: find $\boldsymbol{u}_h \in \boldsymbol{V}_h$
  such that
  \begin{equation}
    \label{eq:hdgvectorPoisson}
    a_h(\boldsymbol{u}_h, \boldsymbol{v}_h)
    + \gamma (\nabla \cdot u_h, \nabla \cdot v_h)_{\mathcal{T}}
    =
    (v_h, f)_{\mathcal{T}} \qquad \forall \boldsymbol{v}_h \in \boldsymbol{V}_h.
  \end{equation}
  Since $\gamma B_{pu}^TM^{-1}B_{pu}$ is the matrix representation of
  $\gamma (\nabla \cdot u_h, \nabla \cdot v_h)_{\mathcal{T}}$ (see
  \cite{Farrell:2021}), the matrix form of \cref{eq:hdgvectorPoisson}
  is given by
  \begin{equation}
    \label{eq:matrixformVectorPoisson}
    \begin{bmatrix}
      A_{uu} + \gamma B_{pu}^TM^{-1}B_{pu} & A_{\bar{u}u}^T \\
      A_{\bar{u}u} & A_{\bar{u}\bar{u}}
    \end{bmatrix}
    \begin{bmatrix}
      u \\ \bar{u}
    \end{bmatrix}
    =
    \begin{bmatrix}
      L_u \\ 0
    \end{bmatrix}.
  \end{equation}
  Eliminating $u$ from \cref{eq:matrixformVectorPoisson} we find the
  following problem for $\bar{u}$:
  \begin{equation}
    \label{eq:matreducedVecPoisson}
    \bar{A}_{\gamma}\bar{u}
    = \bar{F}_{\gamma},
  \end{equation}
  where
  $\bar{F}_{\gamma} = -A_{\bar{u}u}(A_{uu} + \gamma
  B_{pu}^TM^{-1}B_{pu})^{-1}F$.

  Let us now eliminate $u_h$ from \cref{eq:hdgvectorPoisson}. For
  this, let $s \in \sbr[0]{L^2(\Omega)}^d$ and
  $\bar{m}_h \in \bar{V}_h$ be given and define on an element $K$
  \begin{multline}
    a_K(v_h, w_h) :=
    \del{\nabla v_h, \nabla w_h}_{K}
    + \gamma(\nabla \cdot v_h, \nabla \cdot w_h)_{K}
    \\
    - \left\langle \partial_n v_h, w_h \right\rangle_{\partial K}
    - \left\langle v_h, \partial_n w_h \right\rangle_{\partial K}
    + \langle \alpha h^{-1}w_h, v_h \rangle_{\partial K},
  \end{multline}
  and
  \begin{equation}
    L_K(w_h) :=
    (s, w_h)_{K}
    - \left\langle \partial_n w_h, \bar{m}_h \right\rangle_{\partial K}
    + \langle \alpha h^{-1}w_h, \bar{m}_h \rangle_{\partial K}.
  \end{equation}
  Let $v_h^L(\bar{m}_h, s) \in V_h$ be the function such that its
  restriction to element $K$ satisfies the following local problem:
  given $s \in \sbr[0]{L^2(\Omega)}^d$ and $\bar{m}_h \in \bar{V}_h$,
  \begin{equation}
    \label{eq:localproblem}
    a_K(v_h^L, w_h) = L_K(w_h) \quad \forall w_h \in \sbr{P_k(K)}^d.
  \end{equation}
  Suppose now that $\boldsymbol{u}_h \in \boldsymbol{V}_h$ satisfies
  \cref{eq:hdgvectorPoisson} and that $f \in
  \sbr[0]{L^2(\Omega)}^d$. Define
  $l(\bar{u}_h) := v_h^L(\bar{u}_h, 0)$ and $u_h^f := v_h^L(0,
  f)$. Then $u_h = u_h^f + l(\bar{u}_h)$ where
  $\bar{u}_h \in \bar{V}_h$ satisfies
  \begin{equation}
    \label{eq:weakformredVecPoisson}
    a_h((l(\bar{u}_h), \bar{u}_h), (l(\bar{w}_h), \bar{w}_h))
    + \gamma(\nabla \cdot l(\bar{u}_h), \nabla \cdot l(\bar{w}_h))
    = (f, l(\bar{w}_h))_{\mathcal{T}} \quad \forall \bar{w}_h \in \bar{V}_h.
  \end{equation}
  The steps to show that \cref{eq:weakformredVecPoisson} holds are
  identical to the steps in the proof of
  \cite[Lemma~4]{Rhebergen:2018b}, so are omitted here. We next remark
  that \cref{eq:matreducedVecPoisson} is the matrix formulation of
  \cref{eq:weakformredVecPoisson}. Let us define
  \begin{equation}
    \bar{a}_h(\bar{u}_h, \bar{w}_h)
    :=
    a_h((l(\bar{u}_h), \bar{u}_h), (l(\bar{w}_h), \bar{w}_h))
    + \gamma(\nabla \cdot l(\bar{u}_h), \nabla \cdot l(\bar{w}_h)),
  \end{equation}
  then $\bar{A}_{\gamma}$ is the matrix obtained from the
  discretization of $\bar{a}_h(\bar{u}_h, \bar{w}_h)$. To conclude
  this proof we need to show that $\bar{a}_h(\cdot, \cdot)$ is
  equivalent to $\tnorm{\cdot}_h^2$.

  Since
  $\sum_{K \in \mathcal{T}}\norm{\nabla \cdot v_h}_{K} \le
  \tnorm{\boldsymbol{v}_h}_{v}$ it is clear from
  \cref{eq:stab_bound_ah} that
  \begin{equation}
    c_a^s\tnorm{\boldsymbol{v}_h}_v^2
    \le
    a_h(\boldsymbol{v}_h, \boldsymbol{v}_h)
    + \gamma \sum_{K \in \mathcal{T}} \norm{\nabla \cdot v_h}_{K}
    \le
    (c_a^b + \gamma) \tnorm{\boldsymbol{v}_h}_{v}.
  \end{equation}
  Following now identical steps as in the proof of \cite[Lemma
  5]{Rhebergen:2018b} (we omit these steps), there exist positive
  constants $C_1$ and $C_2$ independent of $h_K$ such that
  \begin{equation}
    \label{eq:H1equiv}
    C_1 \tnorm{\bar{w}_h}^2_h
    \le
    \bar{a}_h(\bar{w}_h, \bar{w}_h)
    \le
    C_2(1 + \gamma) \tnorm{\bar{w}_h}^2_h,
  \end{equation}
  so that the result follows.
\end{proof}

\Cref{lem:Agamma-H1equiv} shows that $\bar{A}_{\gamma}$ is equivalent to
an $H^1$-operator for finite values of $\gamma$, hence, standard
algebraic multigrid is expected to be effective on $\bar{A}_{\gamma}$.
\Cref{lem:Agamma-H1equiv} also shows that the larger we choose $\gamma$,
i.e., the better $\bar{A}_{\gamma}$ approximates $\bar{A}^d$, the weaker
the equivalence between $\bar{A}_{\gamma}$ and the $H^1$-operator. We
therefore expect multigrid to be less effective for large values of
$\gamma$. In our numerical simulations we therefore choose small values
for this parameter.

\section{Numerical examples}
\label{sec:num_examples}

We now examine the performance of MINRES combined with the
preconditioners $\mathbb{P}_{\bar{M}}$ and $\mathbb{P}_{BAB}$ introduced
in \cref{thm:P2x2precon}. All examples in this section have been
implemented in MFEM~\cite{mfem-library} and we use the PETSc
\cite{petsc-user-ref, petsc-web-page} implementation of MINRES.

In the implementation of the application of $\mathbb{P}_{\bar{M}}$ and
$\mathbb{P}_{BAB}$, unless specified differently, we use a direct
solver to compute $\bar{M}^{-1}$ and
$(B_{\bar{p}u}A_{uu}^{-1}B_{\bar{p}u}^T)^{-1}$. Furthermore, we
consider different choices for $(\bar{R}^d)^{-1}$ as will be discussed
below. When we apply multigrid, we use classical algebraic multigrid
(four multigrid V-cycles) with one application of a symmetric
Gauss--Seidel smoother (pre and post) via the BoomerAMG
library~\cite{Henson:2002}. In all examples the MINRES iterations are
terminated once the relative true residual reaches a tolerance of
$10^{-8}$. We consider unstructured simplicial meshes in two and three
dimensions.

Let $k$ be the polynomial degree in our function space
$\boldsymbol{X}_h$.  In two dimensions we set the penalty parameter
$\alpha$ to $\alpha=4k^2$ while it is set to $\alpha=6k^2$ in three
dimensions. We compare the performance of $\mathbb{P}_{\bar{M}}$ and
$\mathbb{P}_{BAB}$ also to the performance of
$\mathbb{P}_{3\times 3}$, a preconditioner we presented previously in
\cite{Rhebergen:2018b}. Where $\mathbb{P}_{\bar{M}}$ and
$\mathbb{P}_{BAB}$ are preconditioners for the two-field reduced
system \cref{eq:condensed2form_up} in which both $u$ and $p$ have been
eliminated from \cref{eq:fullmatrix}, $\mathbb{P}_{3\times 3}$ is a
preconditioner for the three-field reduced system
\cref{eq:condensed3form} in which only $u$ has been eliminated from
\cref{eq:fullmatrix}. For convenience, we summarize the $3 \times 3$
preconditioner in~\cref{ss:elim_vel}.

\subsection{Optimality assessment}
\label{ss:optimality}

We examine the number of required MINRES iterations with mesh
refinement for the two-dimensional lid-driven cavity problem for
polynomial degrees $k=2$, $k=3$ and $k=4$. In particular, we consider
the square domain $\Omega := [-1, 1]^2$ and impose the Dirichlet
boundary condition $u = (1-x_1^4, 0)$ on the boundary with $x_2 = 1$,
and $u=0$ on the remaining boundaries.

We first compare the performance of $\mathbb{P}_{\bar{M}}$ and
$\mathbb{P}_{BAB}$ when making the following choices for
$(\bar{R}^d)^{-1}$:
\begin{equation}
  \label{eq:Rd_LU}
  (\bar{R}^d)^{-1} = (\bar{A}^d)^{-1}
  \quad\text{and}\quad
  (\bar{R}^d)^{-1} = (\bar{A}_{\gamma})^{-1},
\end{equation}
and with $\gamma = 0$ and $\gamma = 0.1$. The number of iterations
required for MINRES to reach convergence is listed in
\cref{tab:preconditioning_23d_op_LU}. We draw the following conclusions
from this table:
\begin{itemize}
\item Choosing $(\bar{R}^d)^{-1} = (\bar{A}^d)^{-1}$ we are guaranteed
  that $(\bar{R}^d)^{-1}$ is spectrally equivalent to
  $(\bar{A}^d)^{-1}$. With this choice we observe that the iteration
  count for MINRES to converge to a given tolerance is independent of
  $h$. This verifies \cref{thm:P2x2precon} that $\mathbb{P}_{\bar{M}}$
  and $\mathbb{P}_{BAB}$ are optimal preconditioners.
\item We observe that choosing
  $(\bar{R}^d)^{-1}=(\bar{A}_{\gamma})^{-1}$, with $\gamma=0$ and
  $\gamma = 0.1$, result in optimal preconditioners
  $\mathbb{P}_{\bar{M}}$ and $\mathbb{P}_{BAB}$. This verifies that
  $\bar{A}_{\gamma}$ for small $\gamma$ is a good approximation to
  $\bar{A}^d$ (as discussed in \cref{ss:char_barRd}).
\item We observe that as the polynomial degree $k$ increases, there is
  little variation in iteration count when choosing
  $(\bar{R}^d)^{-1} = (\bar{A}^d)^{-1}$, however, there is a slight
  increase in iteration count when choosing
  $(\bar{R}^d)^{-1}=(\bar{A}_{\gamma})^{-1}$. Furthermore, when using
  $\bar{A}_{\gamma}$ to approximate $\bar{A}^d$, the iteration count
  using $\mathbb{P}_{BAB}$ as preconditioner is less than when using
  $\mathbb{P}_{\bar{M}}$ as preconditioner when $k=3$ and $k=4$.
\end{itemize}

A direct solver for inverting the blocks in \cref{eq:Rd_LU} is
prohibitively expensive in large simulations. We therefore next compare
the performance of $\mathbb{P}_{\bar{M}}$ and $\mathbb{P}_{BAB}$ when
replacing the direct solves by multigrid, i.e., making the following
choices for~$(\bar{R}^d)^{-1}$:
\begin{equation}
  \label{eq:Rd_MG}
  (\bar{R}^d)^{-1} = (\bar{A}^d)^{MG}
  \quad\text{and}\quad
  (\bar{R}^d)^{-1} = (\bar{A}_{\gamma})^{MG}.
\end{equation}
Again we consider $\gamma = 0$ and $\gamma = 0.1$. The number of
iterations required for MINRES to reach convergence are listed in
\cref{tab:preconditioning_23d_op_MG}. We draw the following
conclusions from this table:
\begin{itemize}
\item From \cref{tab:preconditioning_23d_op_LU} we observe that
  $(\bar{R}^d)^{-1} = (\bar{A}^d)^{-1}$ is the best choice, resulting in
  a solver that converges in three times fewer iterations than the next
  best choice. However, computing $(\bar{A}^d)^{-1}$ is costly. If we
  replace $(\bar{A}^d)^{-1}$ by $(\bar{A}^d)^{MG}$ we observe in
  \cref{tab:preconditioning_23d_op_MG} that iteration count to
  convergence grows as $h$ decreases, both for $\mathbb{P}_{\bar{M}}$
  and for $\mathbb{P}_{BAB}$. This is as expected from
  \cref{lem:properties-barAd}, in particular \cref{eq:propertyAd_2},
  which shows that $\bar{A}^d$ may not be an $H^1$-like operator. We
  therefore cannot expect standard multigrid to perform well on
  $\bar{A}^d$.
\item Since $\bar{A}_{\gamma} \to \bar{A}$ as $\gamma \to 0$, we expect
  multigrid to be effective on $\bar{A}_{\gamma}$ when $\gamma$ is
  small. This is because $\bar{A}$ is an $H^1$-like operator on which
  multigrid is expected to perform well (see \cref{ss:char_barRd}). This
  is verified here by the observation that the choice $(\bar{R}^d)^{-1}
  = (\bar{A}_{\gamma})^{MG}$, with $\gamma = 0$ and $\gamma = 0.1$,
  results in an optimal preconditioner for both $\mathbb{P}_{\bar{M}}$
  and $\mathbb{P}_{BAB}$.
\item As we saw also in \cref{tab:preconditioning_23d_op_LU}, when
  using $\bar{A}_{\gamma}$ to approximate $\bar{A}^d$, there is a
  slight increase in iteration count as $k$ increases. Finally, we
  observe that when $k=3$ and $k=4$, the iteration count using
  $\mathbb{P}_{BAB}$ as preconditioner is less than when using
  $\mathbb{P}_{\bar{M}}$ as preconditioner.
\end{itemize}

\begin{table}
  \caption{Iteration counts for preconditioned MINRES for the relative
    true residual to reach a tolerance of $10^{-8}$ for the lid-driven
    cavity problem for different polynomial degrees. We compare
    different choices for $(\bar{R}^d)^{-1}$, see \cref{eq:Rd_LU}, in
    $\mathbb{P}_{\bar{M}}$ and $\mathbb{P}_{BAB}$. The test case is
    described in \cref{ss:optimality}.} {\small
    \begin{center}
      \begin{tabular}{c|c||c|c|c||c|c|c}
        \multicolumn{8}{c}{$k=2$}\\ 
        \hline
        \multicolumn{2}{c||}{}
        & \multicolumn{3}{c||}{$\mathbb{P}_{\bar{M}}$}
        & \multicolumn{3}{c}{$\mathbb{P}_{BAB}$} \\
        \hline
        Elements & DOFs & $(\bar{A}^d)^{-1}$ & $(\bar{A}_{0})^{-1}$ & $(\bar{A}_{0.1})^{-1}$ & $(\bar{A}^d)^{-1}$ & $(\bar{A}_{0})^{-1}$ & $(\bar{A}_{0.1})^{-1}$ \\
        \hline
        176         & \num{2574}   & 31 & 90 & 84 & 54 & 87 & 82 \\
        704         & \num{9900}   & 31 & 94 & 87 & 55 & 92 & 86 \\
        \num{2816}  & \num{38808}  & 31 & 96 & 89 & 54 & 95 & 88 \\
        \num{11264} & \num{153648} & 29 & 93 & 86 & 54 & 95 & 87 \\
        \num{45056} & \num{611424} & 29 & 93 & 86 & 54 & 95 & 87 \\
        \hline
        \multicolumn{8}{c}{$k=3$}\\ 
        \hline
        \multicolumn{2}{c||}{}
        & \multicolumn{3}{c||}{$\mathbb{P}_{\bar{M}}$}
        & \multicolumn{3}{c}{$\mathbb{P}_{BAB}$} \\
        \hline
        Elements & DOFs & $(\bar{A}^d)^{-1}$ & $(\bar{A}_{0})^{-1}$ & $(\bar{A}_{0.1})^{-1}$ & $(\bar{A}^d)^{-1}$ & $(\bar{A}_{0})^{-1}$ & $(\bar{A}_{0.1})^{-1}$ \\
        \hline
        176         & \num{3432}   & 26 & 119 & 114 & 50 & 106 & 102 \\
        704         & \num{13200}  & 26 & 114 & 109 & 49 & 107 & 103 \\
        \num{2816}  & \num{51744}  & 24 & 114 & 109 & 49 & 107 & 103 \\
        \num{11264} & \num{204864} & 24 & 114 & 109 & 47 & 107 & 103 \\
        \num{45056} & \num{815232} & 24 & 106 & 101 & 47 & 101 & 97  \\
        \hline
        \multicolumn{8}{c}{$k=4$}\\ 
        \hline
        \multicolumn{2}{c||}{}
        & \multicolumn{3}{c||}{$\mathbb{P}_{\bar{M}}$}
        & \multicolumn{3}{c}{$\mathbb{P}_{BAB}$} \\
        \hline
        Elements & DOFs & $(\bar{A}^d)^{-1}$ & $(\bar{A}_{0})^{-1}$ & $(\bar{A}_{0.1})^{-1}$ & $(\bar{A}^d)^{-1}$ & $(\bar{A}_{0})^{-1}$ & $(\bar{A}_{0.1})^{-1}$ \\
        \hline
        176         & \num{4290}    & 24 & 144 & 137 & 49 & 129 & 124 \\
        704         & \num{16500}   & 24 & 146 & 139 & 49 & 131 & 125 \\
        \num{2816}  & \num{64680}   & 22 & 137 & 140 & 49 & 131 & 126 \\
        \num{11264} & \num{256080}  & 22 & 137 & 130 & 49 & 128 & 123 \\
        \num{45056} & \num{1019040} & 20 & 137 & 130 & 47 & 125 & 120 \\
        \hline
      \end{tabular}
    \label{tab:preconditioning_23d_op_LU}
  \end{center}
}
\end{table}

\begin{table}
  \caption{Iteration counts for preconditioned MINRES for the relative
    true residual to reach a tolerance of $10^{-8}$ for the lid-driven
    cavity problem for different polynomial degrees. We compare
    different choices for $(\bar{R}^d)^{-1}$, see \cref{eq:Rd_MG}, in
    $\mathbb{P}_{\bar{M}}$ and $\mathbb{P}_{BAB}$. The test case is
    described in \cref{ss:optimality}.}
  {\fontsize{8.5}{10.5}\selectfont
    \begin{center}
      \begin{tabular}{c|c||c|c|c||c|c|c}
        \multicolumn{8}{c}{$k=2$}\\ 
        \hline
        \multicolumn{2}{c||}{}
        & \multicolumn{3}{c||}{$\mathbb{P}_{\bar{M}}$}
        & \multicolumn{3}{c}{$\mathbb{P}_{BAB}$} \\
        \hline
        Elements & DOFs & $(\bar{A}^d)^{MG}$ & $(\bar{A}_{0})^{MG}$ & $(\bar{A}_{0.1})^{MG}$ & $(\bar{A}^d)^{MG}$ & $(\bar{A}_{0})^{MG}$ & $(\bar{A}_{0.1})^{MG}$ \\
        \hline
        176         & \num{2574}   & 37  & 90 & 84 & 60   & 87  & 82 \\
        704         & \num{9900}   & 52  & 94 & 87 & 91   & 92  & 86 \\
        \num{2816}  & \num{38808}  & 91  & 96 & 85 & 163  & 94  & 88 \\
        \num{11264} & \num{153648} & 174 & 93 & 87 & 334  & 95  & 87 \\
        \num{45056} & \num{611424} & 358 & 95 & 88 & \textgreater 500 & 95 & 90 \\
        \hline
        \multicolumn{8}{c}{$k=3$}\\ 
        \hline
        \multicolumn{2}{c||}{}
        & \multicolumn{3}{c||}{$\mathbb{P}_{\bar{M}}$}
        & \multicolumn{3}{c}{$\mathbb{P}_{BAB}$} \\
        \hline
        Elements & DOFs & $(\bar{A}^d)^{MG}$ & $(\bar{A}_{0})^{MG}$ & $(\bar{A}_{0.1})^{MG}$ & $(\bar{A}^d)^{MG}$ & $(\bar{A}_{0})^{MG}$ & $(\bar{A}_{0.1})^{MG}$ \\
        \hline
        176         & \num{3432}   & 38  & 119 & 114 & 65  & 106 & 102 \\
        704         & \num{13200}  & 59  & 114 & 109 & 101 & 107 & 103 \\
        \num{2816}  & \num{51744}  & 101 & 115 & 110 & 193 & 107 & 103 \\
        \num{11264} & \num{204864} & 190 & 116 & 110 & 388 & 107 & 103 \\
        \num{45056} & \num{815232} & 374 & 110 & 105 & \textgreater 500 & 105 & 101 \\
        \hline
        \multicolumn{8}{c}{$k=4$}\\ 
        \hline
        \multicolumn{2}{c||}{}
        & \multicolumn{3}{c||}{$\mathbb{P}_{\bar{M}}$}
        & \multicolumn{3}{c}{$\mathbb{P}_{BAB}$} \\
        \hline
        Elements & DOFs & $(\bar{A}^d)^{MG}$ & $(\bar{A}_{0})^{MG}$ & $(\bar{A}_{0.1})^{MG}$ & $(\bar{A}^d)^{MG}$ & $(\bar{A}_{0})^{MG}$ & $(\bar{A}_{0.1})^{MG}$ \\
        \hline
        176         & \num{4290}    & 45  & 145 & 138 & 76  & 129 & 124 \\
        704         & \num{16500}   & 68  & 147 & 140 & 127 & 131 & 125 \\
        \num{2816}  & \num{64680}   & 121 & 138 & 132 & 258 & 131 & 126 \\
        \num{11264} & \num{256080}  & 230 & 139 & 133 & \textgreater 500 & 129 & 124 \\
        \num{45056} & \num{1019040} & 477 & 142 & 135 & \textgreater 500 & 130 & 124 \\
        \hline
      \end{tabular}
    \label{tab:preconditioning_23d_op_MG}
  \end{center}
}
\end{table}

\subsection{Comparison HDG, EDG, and EDG-HDG}
\label{ss:EDG_EDGHDG}

The analysis of the preconditioner in \cref{thm:P2x2precon} holds also
for the embedded discontinuous Galerkin (EDG) and embedded-hybridized
discontinuous Galerkin (EDG-HDG) discretizations of the Stokes problem
\cite{Rhebergen:2020}. The EDG-HDG discretization is given by
replacing $\bar{V}_h$ by the continuous trace velocity space
$\bar{V}_h \cap C^0(\Gamma_0)$ in \cref{eq:discrete_problem}. In the
EDG method both the trace velocity and trace pressure functions are
continuous; in \cref{eq:discrete_problem} $\bar{V}_h$ and $\bar{Q}_h$
are replaced by, respectively, $\bar{V}_h \cap C^0(\Gamma_0)$ and
$\bar{Q}_h \cap C^0(\Gamma_0)$. Both the EDG and the EDG-HDG
discretizations result in velocity approximations that are exactly
divergence-free on each cell.

We will compare results of the preconditioners $\mathbb{P}_{\bar{M}}$
and $\mathbb{P}_{BAB}$, in which we choose
$(\bar{R}^d)^{-1} = (\bar{A}_{\gamma})^{MG}$, with $\gamma = 0$ and
$\gamma = 0.1$, to results obtained when using the preconditioner
$\mathbb{P}_{3\times 3}$ from~\cite{Rhebergen:2018b}. We will apply
these two preconditioners to HDG, EDG, and EDG-HDG discretizations of
the Stokes problem. We consider the two- and three-dimensional
lid-driven cavity problems. The two-dimensional problem is described
\cref{ss:optimality}. In three dimensions we consider the cube
$\Omega := [0, 1]^3$ and we impose
$u = (1 - \tau_1^4, (1 - \tau_2^4)/10, 0)$ with $\tau_i = 2x_i - 1$ on
the boundary with $x_3 = 1$ and $u = 0$ on the remaining
boundaries. We consider only the case $k=2$.

The number of iterations for MINRES to reach convergence are presented
for the two dimensional problem in \cref{tab:preconditioning_2d} and for
the three dimensional problem in \cref{tab:preconditioning_3d}. From
both tables it is clear that both $\mathbb{P}_{\bar{M}}$ and
$\mathbb{P}_{BAB}$ are optimal preconditioners for HDG, EDG and EDG-HDG
discretizations, i.e.~the iteration count does not grow systematically
with increasing problem size. We also observe that 20-60\% fewer
iterations are required to solve the linear system using the two-field
preconditioners $\mathbb{P}_{\bar{M}}$ and $\mathbb{P}_{BAB}$ compared
to using the three-field preconditioner $\mathbb{P}_{3\times 3}$
from~\cite{Rhebergen:2018b}. Finally, we note that for all calculations,
fewer iterations are required when using $\gamma = 0.1$ than when
using~$\gamma = 0$.

\begin{table}
  \caption{Iteration counts for preconditioned MINRES for the relative
    true residual to reach a tolerance of $10^{-8}$ for the lid-driven
    cavity problem in two dimensions. A comparison between using
    $\mathbb{P}_{\bar{M}}$ and $\mathbb{P}_{BAB}$ with
    $(\bar{R}^d)^{-1} = (\bar{A}_{\gamma})^{MG}$,
    with $\gamma = 0$ and $\gamma = 0.1$, and $\mathbb{P}_{3\times 3}$
    for HDG, EDG, and EDG-HDG discretizations of the Stokes problem.
    The test case is described in \cref{ss:EDG_EDGHDG}. The iteration
    count when using $\gamma = 0.1$ is given in brackets.} {\small
    \begin{center}
      \begin{tabular}{c|cc|cc|cc}
        \multicolumn{1}{c}{} & \multicolumn{6}{c}{$\mathbb{P}_{\bar{M}}$} \\
        \hline
        & \multicolumn{2}{c|}{HDG} & \multicolumn{2}{c|}{EDG} & \multicolumn{2}{c}{EDG-HDG} \\
        Elements & DOFs & Its & DOFs & Its & DOFs & Its \\
        \hline
        176         & \num{2574}   & 90 (84) & \num{1191}   & 86 (81) & \num{1652}   & 95 (86) \\
        704         & \num{9900}   & 94 (87) & \num{4491}   & 83 (79) & \num{6294}   & 92 (88) \\
        \num{2816}  & \num{38808}  & 96 (85) & \num{17427}  & 85 (80) & \num{24553}  & 94 (85) \\
        \num{11264} & \num{153648} & 93 (87) & \num{68643}  & 82 (77) & \num{96978}  & 92 (86) \\
        \num{45056} & \num{611424} & 95 (88) & \num{272451} & 83 (78) & \num{385442} & 88 (83) \\
        \hline
        \multicolumn{1}{c}{} & \multicolumn{6}{c}{$\mathbb{P}_{BAB}$} \\
        \hline
        & \multicolumn{2}{c|}{HDG} & \multicolumn{2}{c|}{EDG} & \multicolumn{2}{c}{EDG-HDG} \\
        Elements & DOFs & Its & DOFs & Its & DOFs & Its \\
        \hline
        176         & \num{2574}   & 87 (82) & \num{1191}   & 61 (56) & \num{1652}   & 61 (58) \\
        704         & \num{9900}   & 92 (86) & \num{4491}   & 62 (59) & \num{6294}   & 64 (59) \\
        \num{2816}  & \num{38808}  & 94 (88) & \num{17427}  & 63 (60) & \num{24553}  & 63 (60) \\
        \num{11264} & \num{153648} & 95 (87) & \num{68643}  & 63 (58) & \num{96978}  & 63 (60) \\
        \num{45056} & \num{611424} & 95 (90) & \num{272451} & 61 (58) & \num{385442} & 63 (60) \\
        \hline
        \multicolumn{1}{c}{} & \multicolumn{6}{c}{$\mathbb{P}_{3\times 3}$} \\
        \hline
        & \multicolumn{2}{c|}{HDG} & \multicolumn{2}{c|}{EDG} & \multicolumn{2}{c}{EDG-HDG} \\
        Elements & DOFs & Its & DOFs & Its & DOFs & Its \\
        \hline
        176         & \num{3102}   & 130 & \num{1719}   & 124 & \num{2180}   & 131 \\
        704         & \num{12012}  & 134 & \num{6603}   & 127 & \num{8406}   & 136 \\
        \num{2816}  & \num{47256}  & 136 & \num{25875}  & 124 & \num{33002}  & 131 \\
        \num{11264} & \num{187440} & 138 & \num{102435} & 125 & \num{130770} & 135 \\
        \num{45056} & \num{746592} & 132 & \num{407619} & 118 & \num{520610} & 127 \\
        \hline
      \end{tabular}
    \label{tab:preconditioning_2d}
  \end{center}
}
\end{table}

\begin{table}
  \caption{Iteration counts for preconditioned MINRES for the relative
    true residual to reach a tolerance of $10^{-8}$ for the lid-driven
    cavity problem in three dimensions. A comparison between using
    $\mathbb{P}_{\bar{M}}$ and $\mathbb{P}_{BAB}$ with
    $(\bar{R}^d)^{-1} = (\bar{A}_{\gamma})^{MG}$,
    with $\gamma = 0$ and $\gamma = 0.1$, and $\mathbb{P}_{3\times 3}$
    for HDG, EDG, and EDG-HDG discretizations of the Stokes problem.
    The test case is described in \cref{ss:EDG_EDGHDG}. The iteration
    count when using $\gamma = 0.1$ is given in brackets.} {\small
    \begin{center}
      \begin{tabular}{c|cc|cc|cc}
        \multicolumn{1}{c}{} & \multicolumn{6}{c}{$\mathbb{P}_{\bar{M}}$} \\
        \hline
        & \multicolumn{2}{c|}{HDG} & \multicolumn{2}{c|}{EDG} & \multicolumn{2}{c}{EDG-HDG} \\
        Elements & DOFs & Its & DOFs & Its & DOFs & Its \\
        \hline
        524         & \num{28032}    & 172 (160) & \num{3884}   & 99 (94) & \num{9921}   & 150 (143) \\
        \num{4192}  & \num{212736}   & 161 (149) & \num{26452}  & 91 (86) & \num{73023}  & 133 (121) \\
        \num{33536} & \num{1655808}  & 151 (140) & \num{194724} & 80 (75) & \num{559995} & 117 (94) \\
        \hline
        \multicolumn{1}{c}{} & \multicolumn{6}{c}{$\mathbb{P}_{BAB}$} \\
        \hline
        & \multicolumn{2}{c|}{HDG} & \multicolumn{2}{c|}{EDG} & \multicolumn{2}{c}{EDG-HDG} \\
        Elements & DOFs & Its & DOFs & Its & DOFs & Its \\
        \hline
        524         & \num{28032}    & 139 (129) & \num{3884}   & 62 (60) & \num{9921}   & 70 (67) \\
        \num{4192}  & \num{212736}   & 142 (132) & \num{26452}  & 61 (58) & \num{73023}  & 66 (63) \\
        \num{33536} & \num{1655808}  & 139 (129) & \num{194724} & 57 (55) & \num{559995} & 60 (58) \\
        \hline
        \multicolumn{1}{c}{} & \multicolumn{6}{c}{$\mathbb{P}_{3\times 3}$} \\
        \hline
        & \multicolumn{2}{c|}{HDG} & \multicolumn{2}{c|}{EDG} & \multicolumn{2}{c}{EDG-HDG} \\
        Elements & DOFs & Its & DOFs & Its & DOFs & Its \\
        \hline
        524         & \num{30128}    & 224 & \num{5980}   & 151 & \num{12017}  & 223 \\
        \num{4192}  & \num{229504}   & 218 & \num{43220}  & 146 & \num{89791}  & 206 \\
        \num{33536} & \num{1789952}  & 209 & \num{328868} & 134 & \num{694139} & 176 \\
        \hline
      \end{tabular}
    \label{tab:preconditioning_3d}
  \end{center}
}
\end{table}

\subsection{The effect of viscosity}
\label{ss:viscosity}

Consider the following modification of the Stokes problem
\cref{eq:stokes}:
\begin{subequations}
  \label{eq:stokes_nu}
  \begin{align}
    -\nu \nabla^2u + \nabla p &= f & & \text{in } \Omega,
    \\
    \nabla\cdot u &= 0 & & \text{in } \Omega,
    \\
    u &= u_d & & \text{on } \partial\Omega,
  \end{align}
\end{subequations}
where $\nu > 0$ denotes a constant viscosity and $u_d$ a given Dirichlet boundary condition. The discrete formulation
for this problem is given by:
\begin{equation}
  \label{eq:discrete_problem_nu}
  \nu a_h(\boldsymbol{u}_h, \boldsymbol{v}_h)
  + b_h(\boldsymbol{p}_h, v_h)
  + b_h(\boldsymbol{q}_h, u_h)
  = \del{v_h, f}_{\mathcal{T}}
  \quad \forall (\boldsymbol{v}_h, \boldsymbol{q}_h) \in \boldsymbol{X}_h,
\end{equation}
where $a_h(\cdot, \cdot)$ and $b_h(\cdot, \cdot)$ are defined in
\cref{eq:bilin_forms}. Redefining the matrix $A$ in \cref{eq:matrix_A}
as $A \leftarrow \nu A$, the block matrix form of
\cref{eq:discrete_problem_nu} is given by \cref{eq:fullmatrix} and its
statically condensed form is given by
\cref{eq:condensed2form_up}. Redefining furthermore the element and
trace pressure mass matrices in \cref{eq:def_cell_facet_mmatrix} as,
respectively, $M \leftarrow \nu^{-1}M$ and
$\bar{M} \leftarrow \nu^{-1}\bar{M}$, the preconditioners again take
on the forms presented in \cref{sec:preconditioning}.

In this section we consider the effect of the viscosity parameter on
solving the statically condensed problem. As test case we choose the
source term $f$ and the Dirichlet boundary condition in
\cref{eq:stokes_nu} such that the exact solution on the domain
$\Omega := [-1, 1]^2$ is given by
\begin{equation}
  u=
  \begin{bmatrix}
    \sin(\pi x_1)\sin(\pi x_2) \\
    \cos(\pi x_1)\cos(\pi x_2)
  \end{bmatrix},
  \
  p = \sin(\pi x_1)\cos(\pi x_2).
\end{equation}
In our simulations we choose polynomial degree $k=2$.

In \cref{tab:performance_nu} we present the number of iterations for
MINRES to reach convergence for $\nu = 1$ and $\nu = 10^{-6}$ for the
HDG, EDG, and EDG-HDG discretizations. We also present the $L^2$-norm
of the velocity error, pressure error, the cell-wise divergence error
as well as $\max_{q \in \mathcal{Q}}|\jump{u \cdot n}|$, where
$\mathcal{Q}$ is the set of all quadrature points on the facets.

The HDG and EDG-HDG discretizations result in a velocity field $u_h$
that is pointwise divergence-free on the cells and
$H(\text{div})$-conforming. A consequence of these properties is that
the magnitude of the viscosity does not affect the $L^2$-norm of the
velocity error \cite[Theorem 3]{Rhebergen:2020}. This
`pressure-robustness' is also observed in \cref{tab:performance_nu};
the $L^2$-norm of the error of the velocity for $\nu=1$ and
$\nu=10^{-6}$ are more or less identical. (Note that when viscosity
decreases, the conditioning of the matrix worsens, explaining the
increase in the error in the divergence of the velocity and the error
in the jump of the normal component of the velocity across facets with
decreasing viscosity. This, however, does not affect the
pressure-robustness of the discretization.)

The EDG method, on the other hand, results in a velocity field $u_h$
that is pointwise divergence-free on the cells, but not
$H(\text{div})$-conforming. As a consequence, the upper bound for the
$L^2$-norm of the error of the velocity is inversely proportional to
the viscosity (see also \cite[Remark 1]{Rhebergen:2020}). We indeed
observe this in \cref{tab:performance_nu};
$\norm{\nabla \cdot u_h}_{\Omega}$ is close to machine precision while
the error in $\max_{q \in \mathcal{Q}}|\jump{u \cdot n}|$ is
magnitudes larger. Furthermore, we observe an increase in the
$L^2$-norm of the velocity error when viscosity is decreased.

Finally, we observe for all discretizations that the viscosity has no
effect on the number of iterations for MINRES to reach convergence
when using the preconditioner $\mathbb{P}_{BAB}$. When using the
preconditioner $\mathbb{P}_{\bar{M}}$ there is a slight increase of
9-14\% in iteration count as the viscosity decreases from $\nu = 1$ to
$\nu = 10^{-6}$.

\begin{table}
  \caption{$L^2$-norm of the velocity error, pressure error, the
    cell-wise divergence error, the error in the jump of the normal
    component of the velocity across facets, and iteration counts for
    preconditioned MINRES for the relative preconditioned residual
    norm to reach a tolerance of $10^{-12}$ for the two dimensional
    test case described in \cref{ss:viscosity}. By
    $\mathbb{P}_{\bar{M}}^{\gamma=0}$ we mean that we solve
    \cref{eq:condensed2form_up} using MINRES with preconditioner
    $\mathbb{P}_{\bar{M}}$ in which we choose
    $(\bar{R}^d)^{-1} = (\bar{A}_{\gamma})^{MG}$, with $\gamma=0$. The
    other methods are described similarly. The numbers listed with an
    asterisk are the number of iterations required for the relative
    preconditioned residual norm to reach a tolerance of
    $10^{-11}$. The relative preconditioned residual stagnates shortly
    after unable to reach $10^{-12}$. This happens only for $\nu=1$ for
    the EDG and EDG-HDG discretizations on the coarsest mesh.}
    \centering {\fontsize{3.8}{10}\selectfont
  \begin{tabular}{cc|cccc|cc|cc}
    Elements & DOFs & $\norm{u-u_h}_{\Omega}$ & $\norm{p-p_h}_{\Omega}$ & $\norm{\nabla\cdot u_h}_{\Omega}$ & $|\jump{u_h\cdot n}|$ & $\mathbb{P}_{\bar{M}}^{\gamma=0}$ & $\mathbb{P}_{BAB}^{\gamma=0}$ & $\mathbb{P}_{\bar{M}}^{\gamma=0.1}$ & $\mathbb{P}_{BAB}^{\gamma=0.1}$ \\
    \hline
    \multicolumn{10}{c}{HDG, $\nu = 1$} \\
    \hline
    \num{11264}  & \num{153648}  & 6.1e-6 & 2.4e-3 & 1.1e-13 & 9.0e-14 & 157 & 148 & 146 & 140 \\
    \num{45056}  & \num{611424}  & 7.6e-7 & 5.9e-4 & 2.2e-13 & 9.4e-14 & 159 & 152 & 148 & 142 \\
    \num{180224} & \num{2439360} & 9.5e-8 & 1.5e-4 & 4.5e-13 & 5.5e-14 & 158 & 155 & 146 & 146 \\
    \hline
    \multicolumn{10}{c}{HDG, $\nu = 10^{-6}$} \\
    \hline
    \num{11264}  & \num{153648}  & 6.1e-6 & 4.4e-4 & 1.0e-10 & 2.3e-9  & 171 & 150 & 159 & 140 \\
    \num{45056}  & \num{611424}  & 7.8e-7 & 1.1e-4 & 1.0e-10 & 4.3e-10 & 173 & 154 & 161 & 145 \\
    \num{180224} & \num{2439360} & 1.4e-7 & 2.7e-5 & 1.1e-10 & 1.4e-9  & 172 & 157 & 162 & 148 \\
    \hline
    \multicolumn{10}{c}{EDG, $\nu = 1$} \\
    \hline
    \num{11264}  & \num{68643}   & 1.1e-5 & 5.2e-3 & 1.1e-13 & 3.0e-7 & $136^*$ & $91^*$ & $129^*$ & $87^*$ \\
    \num{45056}  & \num{272451}  & 1.4e-6 & 1.3e-3 & 2.2e-13 & 1.9e-8 & 155 & 101 & 146 & 96 \\
    \num{180224} & \num{1085571} & 1.7e-7 & 3.3e-4 & 4.5e-13 & 1.2e-9 & 156 & 102 & 147 & 97 \\
    \hline
    \multicolumn{10}{c}{EDG, $\nu = 10^{-6}$} \\
    \hline
    \num{11264}  & \num{68643}   & 7.6e-3 & 4.4e-4 & 1.0e-10 & 3.8e-4 & 170 & 101 & 160 & 97  \\
    \num{45056}  & \num{272451}  & 4.8e-4 & 1.1e-4 & 1.0e-10 & 1.2e-5 & 176 & 102 & 165 & 100 \\
    \num{180224} & \num{1085571} & 3.0e-5 & 2.7e-5 & 1.1e-10 & 3.7e-7 & 178 & 105 & 168 & 101 \\
    \hline
    \multicolumn{10}{c}{EDG-HDG, $\nu = 1$} \\
    \hline
    \num{11264}  & \num{96978}   & 1.1e-5 & 5.3e-3 & 1.1e-13 & 1.5e-12 & $145^*$ & $93^*$ & $137^*$ & $89^*$ \\
    \num{45056}  & \num{385442}  & 1.4e-6 & 1.3e-3 & 2.2e-13 & 2.3e-13 & 164 & 103 & 155 & 98 \\
    \num{180224} & \num{1536834} & 1.7e-7 & 3.4e-4 & 4.4e-13 & 1.3e-13 & 166 & 104 & 157 & 99 \\
    \hline
    \multicolumn{10}{c}{EDG-HDG, $\nu = 10^{-6}$} \\
    \hline
    \num{11264}  & \num{96978}   & 1.1e-5 & 4.4e-4 & 1.0e-10 & 2.8e-9  & 184 & 102 & 174 & 97  \\
    \num{45056}  & \num{385442}  & 1.4e-6 & 1.1e-4 & 1.0e-10 & 9.3e-10 & 187 & 104 & 177 & 100 \\
    \num{180224} & \num{1536834} & 1.9e-7 & 2.7e-5 & 1.1e-10 & 3.8e-10 & 189 & 105 & 179 & 101 \\
    \hline
  \end{tabular}
  }
  \label{tab:performance_nu}
\end{table}

\subsection{Performance comparison}
\label{ss:performancecomp}

In this final section we compare the performance of
$\mathbb{P}_{\bar{M}}$ and $\mathbb{P}_{BAB}$, in which we choose
$(\bar{R}^d)^{-1} = (\bar{A}_{\gamma})^{MG}$, with $\gamma = 0$ and
$\gamma = 0.1$, to the performance of the $\mathbb{P}_{3\times 3}$
preconditioner. We do this for the HDG, EDG, and EDG-HDG
discretiations. As test case we consider the Stokes problem on the
three-dimensional domain $\Omega = [-1, 1]^3$, with Dirichlet boundary
conditions and a source term such that the exact solution to the
Stokes problem is given by
\begin{equation}
  u=\pi
  \begin{bmatrix}
    \sin(\pi x_1)\cos(\pi x_2) - \sin(\pi x_1)\cos(\pi x_3) \\
    \sin(\pi x_2)\cos(\pi x_3) - \sin(\pi x_2)\cos(\pi x_1) \\
    \sin(\pi x_3)\cos(\pi x_1) - \sin(\pi x_3)\cos(\pi x_2)
  \end{bmatrix},
  \
  p = \sin(\pi x_1)\sin(\pi x_2)\sin(\pi x_3) - \tfrac{8}{\pi^3}.
\end{equation}

In the implementation of the application of $\mathbb{P}_{\bar{M}}$ and
$\mathbb{P}_{BAB}$ we now approximate $\bar{M}^{-1}$ by
$(\bar{M})^{MG}$ and $(B_{\bar{p}u}A_{uu}^{-1}B_{\bar{p}u}^T)^{-1}$ by
$(B_{\bar{p}u}A_{uu}^{-1}B_{\bar{p}u}^T)^{MG}$ as this is slightly
more efficient than using a direct solver for these terms.

In \cref{tab:performance} we list the $L^2$-norm of the velocity
error, pressure error, the cell-wise divergence of the velocity, and
maximum value of the jump of the normal component of the velocity
across facets. These errors are computed on a mesh consisting of
$\num{33536}$ tetrahedra. We also include the CPU time to convergence.

Considering first the error in the $L^2$-norm, for all discretizations
we observe that the $L^2$-norm of the error in element velocity $u_h$
and element pressure $p_h$ are identical when using the two-field
\cref{eq:condensed2form_up} and three-field \cref{eq:condensed3form}
reduced systems. The error in the divergence of the element velocity,
however, is different. In the three-field reduced formulation the
divergence of the element velocity is of the order of accuracy at
which the MINRES method was terminated. This implies that would we
want $\nabla\cdot u_h$ on elements to be of the order of machine
precision, the stopping criteria of the MINRES method would need to be
of the order of machine precision. On the other hand, the error in the
divergence of the element velocity is of the order of machine
precision when using the two-field reduced formulation. This is due to
the element-wise projection $\mathcal{P}$ in
\cref{eq:projection}. Indeed, as we saw previously in
\cref{ss:matrixformulation},
$u = \mathcal{P}A_{uu}^{-1}(L_u - A_{\bar{u}u}^T\bar{u} -
B_{\bar{p}u}^T\bar{p}) \in \Ker\, B_{pu}$. Furthermore, for all
discretizations considered here, we have that $u \in \Ker\, B_{pu}$
implies $\nabla \cdot u_h = 0$ pointwise on each element. The error in
the divergence of the velocity therefore does not depend on the
stopping criteria used for MINRES. Finally, we note that the error in
the jump of the normal component of the velocity using the EDG
discretization is magnitudes larger than when using HDG or EDG-HDG
(which are of the order of the stopping criteria of the MINRES
method). This is independent of whether the two-field or three-field
reduced systems are solved and independent of which preconditioner is
used. The higher error in the EDG method is expected as it is the only
discretization that is not pressure-robust \cite{Rhebergen:2020}.

Considering now the CPU time and number of iterations required for
convergence, we observe for all discretizations that using $\gamma
= 0.1$ in $\mathbb{P}_{\bar{M}}$ and $\mathbb{P}_{BAB}$ results in fewer
iterations than when using $\gamma = 0$. However, the CPU time to
convergence using $\gamma = 0.1$ is greater than when using $\gamma =
0$. This is in part due to the construction of
$\widehat{A}_{\gamma}^{-1}$ in \cref{eq:Am_definition}, slowing down the
construction of the preconditioner.

Consider now the case when $\gamma = 0$ in $\mathbb{P}_{\bar{M}}$ and
$\mathbb{P}_{BAB}$. We then observe for the HDG discretization that both
two-field preconditioners outperform the three-field preconditioner
$\mathbb{P}_{3 \times 3}$ both in number of iterations (up to 27\% fewer
iterations) and CPU time (up to 32\% faster). For the EDG and EDG-HDG
discretizations the performance is even better; note, for example, that
using $\mathbb{P}_{\bar{M}}$ for the EDG discretization is up to 43\%
faster and $\mathbb{P}_{BAB}$ is up to 51\% faster than using the
$\mathbb{P}_{3 \times 3}$ preconditioner.

\begin{table}
  \caption{ Errors in the $L^2$-norm, CPU times and iteration counts
    for preconditioned MINRES for the relative preconditioned residual
    norm to reach a tolerance of $10^{-8}$ for the three dimensional
    test case described in \cref{ss:performancecomp}. Here HDG
    $\mathbb{P}_{3\times 3}$ means that we solve
    \cref{eq:condensed3form} using MINRES with preconditioner
    $\mathbb{P}_{3\times 3}$. By HDG $\mathbb{P}_{\bar{M}}^{\gamma=0}$
    we mean that we solve \cref{eq:condensed2form_up} using MINRES
    with preconditioner $\mathbb{P}_{\bar{M}}$ in which we choose
    $(\bar{R}^d)^{-1} = (\bar{A}_{\gamma})^{MG}$, with $\gamma=0$. The
    other methods are described similarly. The results, except
    $\norm{\nabla\cdot u_h}_{\Omega}$ and $|\jump{u_h\cdot n}|$, have
    been normalized with respect to the HDG method using
    $\mathbb{P}_{3\times 3}$ (in brackets).}
    \centering {\fontsize{5.5}{10}\selectfont
  \begin{tabular}{c|cccc|ccc}
    Method & $\norm{u-u_h}_{\Omega}$ & $\norm{p-p_h}_{\Omega}$ & $\norm{\nabla\cdot u_h}_{\Omega}$ & $|\jump{u_h\cdot n}|$ & CPU & Its & DOFs \\
    \hline
    HDG $\mathbb{P}_{3\times 3}$ & 1 (2.5e-4) & 1 (3.3e-2) & 2.1e-6 & 2.4e-8 & 1 (1171s) & 1 (176) & 1 (\num{1789952}) \\
    HDG $\mathbb{P}_{\bar{M}}^{\gamma=0}$ & 1 & 1 & 1.7e-13 & 2.5e-8   & 0.68 & 0.73 & 0.93 \\
    HDG $\mathbb{P}_{\bar{M}}^{\gamma=0.1}$ & 1 & 1 & 1.7e-13 & 2.6e-8 & 1.60 & 0.68 & 0.93 \\
    HDG $\mathbb{P}_{BAB}^{\gamma=0}$ & 1 & 1 & 1.7e-13 & 1.4e-8      & 0.82 & 0.76 & 0.93 \\
    HDG $\mathbb{P}_{BAB}^{\gamma=0.1}$ & 1 & 1 & 1.7e-13 & 1.3e-8     & 1.88 & 0.71 & 0.93 \\
    \hline
    EDG $\mathbb{P}_{3\times 3}$          & 1.84 & 3.00 & 4.0e-6 & 1.3e-5  & 0.063  & 0.68 & 0.18 \\
    EDG $\mathbb{P}_{\bar{M}}^{\gamma=0}$   & 1.84 & 3.00 & 1.7e-13 & 1.3e-5 & 0.036  & 0.43  & 0.11 \\
    EDG $\mathbb{P}_{\bar{M}}^{\gamma=0.1}$ & 1.84 & 3.00 & 1.7e-13 & 1.3e-5 & 0.078  & 0.40  & 0.11 \\
    EDG $\mathbb{P}_{BAB}^{\gamma=0}$      & 1.84 & 3.00 & 1.7e-13 & 1.3e-5 & 0.031  & 0.35  & 0.11 \\
    EDG $\mathbb{P}_{BAB}^{\gamma=0.1}$    & 1.84 & 3.00 & 1.7e-13 & 1.3e-5 & 0.069  & 0.34  & 0.11 \\
    \hline
    EDG-HDG $\mathbb{P}_{3\times 3}$          & 1.96 & 4.24 & 3.6e-6  & 6.0e-8 & 0.094 & 0.85 & 0.39 \\
    EDG-HDG $\mathbb{P}_{\bar{M}}^{\gamma=0}$   & 1.96 & 4.24 & 1.7e-13 & 4.9e-8 & 0.054 & 0.56 & 0.31 \\
    EDG-HDG $\mathbb{P}_{\bar{M}}^{\gamma=0.1}$ & 1.96 & 4.24 & 1.8e-13 & 4.9e-8 & 0.11  & 0.53 & 0.31 \\
    EDG-HDG $\mathbb{P}_{BAB}^{\gamma=0}$      & 1.96 & 4.24 & 1.8e-13 & 2.1e-8 & 0.068 & 0.38 & 0.31 \\
    EDG-HDG $\mathbb{P}_{BAB}^{\gamma=0.1}$    & 1.96 & 4.24 & 1.8e-13 & 2.0e-8 & 0.11  & 0.37 & 0.31 \\
    \hline
  \end{tabular}
  }
  \label{tab:performance}
\end{table}

\section{Conclusions}
\label{sec:conclusions}

The linear system of an HDG discretization of the Stokes equations can
efficiently be statically condensed in two ways: (i) eliminating the
degrees-of-freedom associated to the element approximation of the
velocity (the three-field reduced formulation); or (ii) eliminating the
degrees-of-freedom associated to the element approximation of both the
velocity and the pressure (the two-field reduced formulation). In our
previous work we proposed an optimal preconditioner for the three-field
reduced formulation. In this paper we proposed and analyzed a
preconditioner for the two-field reduced formulation. In the two-field
reduced formulation the lifting of the trace velocity to the elements is
algebraically imposed to be divergence-free. Although this complicates
the analysis, it has been shown that the trace pressure Schur complement
is spectrally equivalent to a simple trace pressure mass matrix and we
used this to introduce a new preconditioner. Numerical examples in two
and three dimensions show that the new preconditioner is more efficient
for solving a hybridized discretization of the Stokes problem than our
previous preconditioner.

\bibliographystyle{plainnat}
\bibliography{references}
\appendix
\section{Eliminating the velocity element degrees-of-freedom}
\label{ss:elim_vel}

Instead of eliminating both the velocity and pressure element
degrees-of-freedom from \cref{eq:fullmatrix} it is possible also to
eliminate only the velocity element degrees-of-freedom $u$. This results
in the \emph{three-field reduced system}:
\begin{equation}
  \label{eq:condensed3form}
  \begin{bmatrix}
    \bar{A} & -A_{\bar{u}u}A_{uu}^{-1}B_{pu}^T & -A_{\bar{u}u}A_{uu}^{-1}B_{\bar{p}u}
    \\
    -B_{pu}A_{uu}^{-1}A_{\bar{u}u}^T & -B_{pu}A_{uu}^{-1}B_{pu}^T
        & -B_{pu}A_{uu}^{-1}B_{\bar{p}u}^T
    \\
    -B_{\bar{p}u}A_{uu}^{-1}A_{\bar{u}u}^T & -B_{\bar{p}u}A_{uu}^{-1}B_{pu}^T
        & -B_{\bar{p}u}A_{uu}^{-1}B_{\bar{p}u}^T
  \end{bmatrix}
  \begin{bmatrix}
    \bar{u} \\ p \\ \bar{p}
  \end{bmatrix}
  =
  \begin{bmatrix}
    L_{\bar{u}} - A_{\bar{u}u}A_{uu}^{-1}L_u \\
    -B_{pu}A_{uu}^{-1}L_u \\
    -B_{\bar{p}u}A_{uu}^{-1}L_u
  \end{bmatrix},
\end{equation}
where $\bar{A} = -A_{\bar{u}u}A_{uu}^{-1}A_{\bar{u}u}^T +
A_{\bar{u}\bar{u}}$. Given the trace velocity $\bar{u}$, the element
pressure $p$, and the trace pressure $\bar{p}$, the element velocity $u$
can be computed element-wise in a post-processing step.

In \cite{Rhebergen:2018b} we developed a preconditioner for
\cref{eq:condensed3form}. We proved \cite[Theorem 2]{Rhebergen:2018b}
that
\begin{equation}
  \label{eq:P33}
  \mathbb{P}_{3\times 3}
  =
  \begin{bmatrix}
    \bar{R} & 0 & 0 \\ 0 & M & 0 \\ 0 & 0 & \bar{M}
  \end{bmatrix},
\end{equation}
is an optimal preconditioner for~\cref{eq:condensed3form} provided
$\bar{R}$ is an operator spectrally equivalent to $\bar{A}$. As
discussed in \cref{ss:char_barRd}, $\bar{A}$ is an $H^1$-like operator
motivating the use of multigrid for $\bar{R}^{-1}$, i.e., we set
$\bar{R}^{-1} = (\bar{A})^{MG}$ (see also
\cite[Section~3.3]{Rhebergen:2018b}).

\end{document}